\documentclass[4paper,10pt]{article}
\usepackage[utf8]{inputenc}
\usepackage[T1]{fontenc}
\usepackage{mathtools}
\usepackage{amsmath}
\usepackage{amsthm}
\usepackage{amssymb}
\usepackage[retainorgcmds]{IEEEtrantools}
\usepackage{stackrel}
\usepackage[all]{xy}
\usepackage{setspace}

\usepackage{framed}
\usepackage{marginnote}
\usepackage{stmaryrd}
\usepackage{leftidx}

\theoremstyle{plain}
\newtheorem{thm}{Theorem}
\newtheorem{lem}{Lemma}
\newtheorem{pr}{Proposition}
\newtheorem{cor}{Corollary}
\theoremstyle{definition} 

\theoremstyle{remark} 
\newtheorem{rem}{Remark}
\newtheoremstyle{claim}
  {0.5\topsep}   
  {0.5\topsep}   
  {\itshape}  
  {0pt}       
  {\itshape} 
  {.}         
  {5pt plus 1pt minus 1pt} 
  {#1 #3}          
\theoremstyle{claim}

\let\AA\relax
\newcommand{\AA}{\mathbb{A}}
\newcommand{\afr}{\mathfrak{a}}
\newcommand{\Bc}{\mathcal{B}}

\newcommand{\C}{\mathbb{C}}
\newcommand{\CC}{\mathbb{C}}
\newcommand{\EE}{\mathbb{E}}
\newcommand{\Ec}{\mathcal{E}}

\newcommand{\Gc}{\mathcal{G}}
\newcommand{\gfr}{\mathfrak{g}}
\newcommand{\Ic}{\mathcal{I}}
\newcommand{\Kc}{\mathcal{K}}

\newcommand{\NN}{\mathbb{N}}
\newcommand{\OO}[1]{O\left(#1\right)}

\newcommand{\QQ}{\mathbb{Q}}
\newcommand{\RR}{\mathbb{R}}
\newcommand{\TT}{\mathbb{T}}
\newcommand{\UU}{\mathbb{U}}
\newcommand{\xb}{\mathbf{x}}
\newcommand{\Yc}{\mathcal{Y}}
\newcommand{\yb}{\mathbf{y}}
\newcommand{\Xc}{\mathcal{X}}
\newcommand{\ZZ}{\mathbb{Z}}
\newcommand{\Zc}{\mathcal{Z}}
\newcommand{\Zfr}{\mathfrak{Z}}
\newcommand{\Pb}{\mathbb{P}}
\renewcommand{\Pr}[1]{\mathbb{P}\left[#1\right]}
\newcommand{\E}{\mathbb{E}}
\newcommand{\Me}[1]{\mathbb{E}\left[#1\right]}

\makeatletter
\newsavebox\myboxA
\newsavebox\myboxB
\newlength\mylenA
\newcommand*\xoverline[2][0.8]{
  \sbox{\myboxA}{$\m@th#2$}
  \setbox\myboxB\null
  \ht\myboxB=\ht\myboxA
  \dp\myboxB=\dp\myboxA
  \wd\myboxB=#1\wd\myboxA
  \sbox\myboxB{$\m@th\overline{\copy\myboxB}$}
  \setlength\mylenA{\the\wd\myboxA}
  \addtolength\mylenA{-\the\wd\myboxB}
  \ifdim\wd\myboxB<\wd\myboxA
     \rlap{\hskip 0.5\mylenA\usebox\myboxB}{\usebox\myboxA}
  \else
     \hskip -0.5\mylenA\rlap{\usebox\myboxA}{\hskip 0.5\mylenA\usebox\myboxB}
  \fi}
\makeatother

\newcommand{\xo}[1]{\xoverline{#1}}

\let\Im\relax
\DeclareMathOperator{\Im}{\mathrm{Im}}

\let\Re\relax
\DeclareMathOperator{\Re}{\mathrm{Re}}

\newenvironment{EA}[1]{\begin{IEEEeqnarray*}{#1}}{\end{IEEEeqnarray*}}
\newcommand{\EAy}{\IEEEyesnumber}

\newcommand{\LRI}[4]{\,\leftidx{^{#3}}{#1}{^{#2}_{#4}}}
\newcommand{\Mes}[2]{(1-\mathbb{E}_{#1})\left[#2\right]}

\usepackage{hyperref}
\begin{document}
\title{Local law for the product of independent non-Hermitian matrices with independent entries}
\author{Yuriy Nemish\\
\\ \textit{University of Toulouse, France}}
\date{\today}
\maketitle
\begin{abstract}
We consider products of independent square non-Hermitian random matrices.
More precisely, let $X_1,\ldots,X_n$ be independent $N\times N$ random matrices with independent entries (real or complex with independent real and imaginary parts) with zero mean and variance $\frac{1}{N}$.
Soshnikov-O'Rourke \cite{OrouSosh} and Götze-Tikhomirov \cite{GotzTikh} showed that the empirical spectral distribution of the product of $n$ random matrices with \emph{iid} entries converges to 
\begin{equation}\label{eq:abs0}
  \frac{1}{n\pi}1_{|z|\leq 1}|z|^{\frac{2}{n}-2}dz d\xoverline{z}.
\end{equation}
We prove that if the entries of the matrices $X_1,\ldots,X_n$ satisfy uniform subexponential decay condition, then   in the bulk the convergence of the ESD of  $X_1\cdots X_n$ to \eqref{eq:abs0} holds up to the scale $N^{-1/2+\varepsilon}$.

\end{abstract}
\section{Introduction}
\label{sec:introduction}

In this paper we study the spectrum of the product of non-Hermitian random matrices with independent entries.

The study of the spectrum of non-Hermitian random matrices dates back to 1965, when Ginibre \cite{Gini} calculated the joint density function for the eigenvalues of $N\times N$ non-symmetric random matrix with independent standard Gaussian entries (Ginibre ensembles).
The similar result for the product of independent complex Ginibre matrices was obtained in \cite{AkemBurd} by Akemann and Burda.
One crucial property of random matrices with Gaussian entries is the determinantal structure, using which exact formulas for many important parameters that characterise the distribution of the eigenvalues (such as $k$-point correlation functions) can be obtained.
If the entries of the matrix are not Gaussian, then we usually do not have exact formulas for the distribution of eigenvalues for finite $N$.
Nevertherless, in many cases as $N$ goes to infinity the spectrum of the models with non-Gaussian coefficients behaves similarly to the Ginibre case.
This is known as universality phenomena.
The aim of this article is to show that universality holds for certain local properties of the products of non-Hermitian random matrices.
We now give a brief review of some known universality results for non-Hermitian random matrices.

\emph{Global regime.}
It can be shown using the exact formula for the eigenvalue density from \cite{Gini}, that the empirical spectral measure defined on the eigenvalues of the Ginibre ensemble with entries normalised to have variance $N^{-1}$ converges weakly to the uniform distribution on the unit disk.
The corresponding universality result, known as the Circular Law theorem and proven in a series of papers between 1985 and 2010 (see \cite{TaoVuKris} for the final version), states that if the entries of the matrix are independent with zero mean and variance $N^{-1}$, then the empirical spectral distribution (ESD) converges weakly to the uniform distribution on the unit disk.
The global regime for the products was studied in \cite{GotzTikh} and \cite{OrouSosh} by Götze-Tikhomirov and O'Rourke-Soshnikov, who established that the ESD of the product of $n$ independent non-Hermitian random matrices with normalised entries converges weakly to the $n$th power of the circular law. 
Note, that in \cite{OrouSosh} an additional $2+\varepsilon$-moment assumption was used.

\begin{figure}
  \centering
\includegraphics[height=5cm]{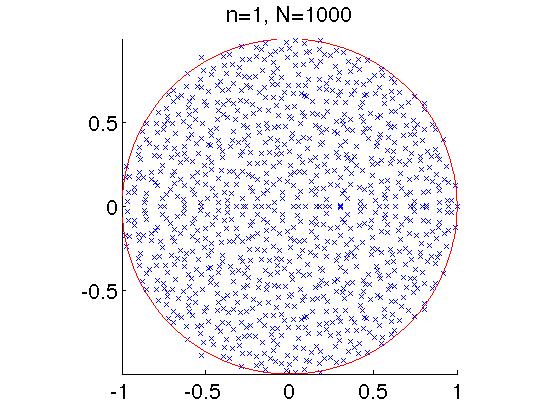}
\includegraphics[height=5cm]{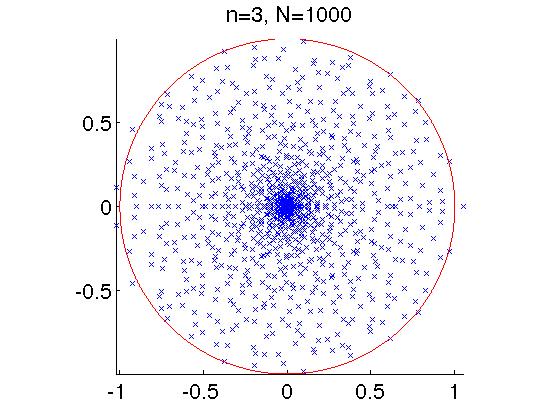}
  \caption{Spectrum of a random Gaussian matrix of size 1000 (left); and spectrum of a product of three independent Gaussian matrices of size 1000 (right).}
\end{figure}

\emph{Intermediate and local regimes.} Global regime deals with weak covergence, which considers the convergence on the subsets containing $cN$ eigenvalues for some $c\geq 0$. 
In other words, we normalize the eigenvalues to have the limiting ESD with compact support.
On the other hand, if we change the normalization of the matrix in such a way, that for any compact set $K$ the number of eigenvalues situated in $K$ is much smaller than $N$, we enter the mesoscopic or itermediate regime.
The smallest scale on which we can expect the linear statistics to have a deterministic behaviour in the limit can be obtained by multiplying the matrix by $\sqrt{N}$.
In this microscopic regime each compact set in the bulk contains only a finite number of eigenvalues.

There has been a remarkable progress recently in the study of universality in the intermediate and local regimes for non-Hermitian matrices.
In \cite{TaoVuNHLoc} Tao and Vu proved universality for the $k$-point correlation functions (see \cite{TaoVuNHLoc} or \cite{PastShch} for definition) under the assumptions that the distributions of the entries of the matrix have exponentially vanishing tails and first four moments matching the moments of Gaussian random variable with zero mean and variance $N^{-1}$.
The last assumption is crucial for \cite{TaoVuNHLoc}, as the approach of Tao and Vu relies on the $4$th moment comparison theorem.

In a series of papers \cite{BourYauYin}, \cite{BourYauYin2} and \cite{Yin} Bourgade, Yau and Yin proved the universality of the local law up to the optimal scale (which can be interpreted as the universality of the $1$-point correlation function) without imposing the $4$th moment matching condition.
The goal of our article is to show a similar result for a product of independent non-Hermitian matrices.
We now introduce some basic objects and fix the notation, that will allow us to state precisely both theorems.

Let $X_1, \ldots,X_n$ be independent $N\times N$ matrices, $X_a=(\LRI{x}{}{a}{ij})_{1\leq i,j\leq N}$, with independent entries (real or complex with independent real and imaginary parts) having zero mean, variance $N^{-1}$ and satisfying the uniform subexponential decay condition
\begin{equation}
  \label{eq:int_2}
  \exists \theta>0
  ,\mbox{ such that }
  \max_{1\leq a \leq n} \max_{1\leq i,j \leq N} \Pb[|\sqrt{N}\LRI{x}{}{a}{ij}|\geq t]\leq \theta^{-1} e^{-t^{\theta}}
  .
\end{equation}

Let $f:\CC\rightarrow \RR_+$ be a smooth non-negative function with compact support, such that $\|f\|_{\infty}\leq C,\|f'\|\leq N^{C}$ for some constant $C>0$.
For any $d\in \RR_+$ and $z_0\in \CC$ we define a $N^{-d}$-rescaling of $f$ around $z_0$ by 
\begin{equation}
  f_{z_0}(z)
  =
  N^{2d}f(N^d(z-z_0))
  .
\end{equation}
For two $N$-dependent random variables $A_N\in \CC$ and $B_N\in \RR_+$ we say that $A$ is stochastically dominated by $B$ (denoted by $A\prec B$) if
\begin{equation}
  \forall D,\varepsilon>0
  \quad
  \Pb[|A_N|\geq N^{\varepsilon} B_N]
  \leq
  N^{-D}
  .
\end{equation}
\begin{thm}[Bourgade-Yau-Yin]
 Let $\mu_1,\ldots,\mu_N$ be the eigenvalues of $X_1$.
Then for any $d\in (0,1/2]$, any $\tau>0$  and $z_0\in \CC$ with $|z_0|\leq \tau^{-1} $
\begin{equation}
  \left(\frac{1}{N}\sum_{j=1}^N f_{z_0}(\mu_j)-\frac{1}{\pi}\int_{|z|< 1}f_{z_0}(z) dz d\xo{z}\right)
  \prec
  N^{-1+2d}\|\Delta f\|_{L_1}
  ,
\end{equation}
where $f_{z_0}$ is the $N^{-d}$-rescaling of $f$ around $z_0$.
\end{thm}

\begin{thm}
Let $\mu_1,\ldots,\mu_N$ be the eigenvalues of $X_1 X_2 \cdots X_n$.
Then for any $d\in (0,1/2]$, any $\tau>0$ small enough and $z_0$ such that $|z_0|\geq \tau$ and $|1-|z_0||\geq \tau$
  \begin{equation}
    \label{eq:int_1}
    \left( \frac{1}{N}\sum_{j=1}^N f_{z_0}(\mu_j)-\frac{1}{n\pi}\int 1_{|z|<1} f_{z_0}(z) |z|^{\frac{2}{n}-2} dz d\xo{z}\right)
    \prec
    N^{-1+2d}\|\Delta f\|_{L_1}
    ,
  \end{equation}
where $f_{z_0}$ is the $N^{-d}$-rescaling of $f$ around $z_0$.
\end{thm}
\begin{rem}
  In the same manner as in \cite{BourYauYin}, \cite{BourYauYin2} and \cite{Yin} we separate the study of the local law in the bulk and at the special points on the edge of the spectrum and at the origin.
In the latter case the analysis of the stability of the self-consistent equations, which is crucial in our approach, cannot be fulfilled, therefore this case requires different tools (for example, ``4th moment comparison''-type results) and is not considered in the present article.
\end{rem}
\begin{rem}
  Recently Ajanki, Erd\"{o}s and Kr\"{u}ger proved that the local law holds up to the optimal scale for a very large class of Hermitian matrices (see \cite{AjanErdoKrug1} and \cite{AjanErdoKrug2}).
Although the model considered in these two articles is very general, it does not contain the matrix $(X-z)^*(X-z)$ studied in the present article, and thus our result cannot be deduced directly from \cite{AjanErdoKrug1} and \cite{AjanErdoKrug2}.
It would be interesting to know how the method of Ajanki, Erd\"{o}s and Kr\"{u}ger can be adjusted in order to obtain the local law for the model considered in the present article.
\end{rem}
\begin{rem}
  Being itself an interesting mathematical problem, the local law on the optimal scale is an important step towards the proof of the universality of the $k$-point correlation functions.
Both known techniques developed to show the local universality (i.e. either using the local relaxation flow or the 4th moment comparison theorem) rely on the initial estimates provided by the local law on the optimal scale.
Therefore, one of the interesting application of the main result of the present article would be proving the universality of the $k$-point correlation functions for the products of non-Hermitian matrices.
\end{rem}

\emph{Outline of the proof.}
We start with the linearization trick, that trasforms the problem about the eigenvalues of the product $X_1\cdots X_n$ into the study of the eigenvalues of a large block matrix $X$ having $X_1,\ldots,X_n$ as blocks.
This will allow us later to exploit the Schur's complement formula to analyse the resolvent matrix.
We show that local law for the product is equivalent to the local circular law for the linerization matrix $X$.
To study the non-Hermitian matrix $X$, we follow Girko's Hermitization techniques, which argues that it is enough to study the distribution of the singular values of the family of shifted matrices $X-z$, $z\in \CC$.
Using the approach developed in \cite{BourYauYin} we show that our initial problem can be reduced to the estimating of the Stieltjes transform of the linearized Hermitized matrix $(X-z)^*(X-z), z\in\CC$.
In Sections~3 and 4 we fix the notation and introduce the tools will be used in the proof of the Stieltjes transform concentration.
The last section is devoted to the study of the Stieltjes transform of the matrix $(X-z)^*(X-z)$.
We adapt the argument of Bourgade-Yau-Yin \cite{BourYauYin} to make it applicable in our setting.
The main difference compared to \cite{BourYauYin} and thus technical difficulty arises from the fact that we cannot work directly with the Stieltjes transform and have to study the concentration for its partial traces.
Similar results but for different values of the resolvent parameter were obtained in \cite{NemiNHPO}.
Although the approach is similar to that used in \cite{NemiNHPO}, many important statements should be adjusted in order to obtain strong enough estimates on a set, which is sufficiently large to imply the rigidity of the singular values of $X-z$.


\section{Reduction to the Stieltjes transform concentration}
\label{sec:}

\emph{Linearisation.}
Following Burda, Janik and Waclaw \cite{BurdJaniWacl} we introduce a block cyclic matrix

\begin{equation}\label{eq:linearisation_matrix}
  X=
  \begin{pmatrix}
    0 & X_1 & 0 &  \cdots & 0 \\
    0 & 0 & X_2 &  \cdots & 0 \\
     & & &\ddots &  \\
    0 & 0 & 0 & \cdots & X_{n-1}\\
    X_n& 0  & 0 &\cdots & 0
  \end{pmatrix}  .
\end{equation}
The $n$-th power of matrix $X$ is an $nN\times nN$ block-diagonal matrix with matrices $X_{a+1}X_{a+2}\cdots X_{a+n},a\in \ZZ/n\ZZ$ on the diagonal. 
The advantage of considering this matrix is that the entries of this matrix are independent with zero mean.
Also, we can rewrite \eqref{eq:int_1} in terms of the eigenvalues of $X$
\begin{EA}{rl}
    \frac{1}{N}\sum_{j=1}^N f_{z_0}(\mu_j)-\frac{1}{n\pi}\int_{|z|<1}f_{z_0}(z) &|z|^{\frac{2}{n}-2}dz d\xo{z}
    \\\EAy \label{eq:red_1}
    &=
    \frac{1}{nN}\sum_{j=1}^{nN} f_{z_0}(\mu^n_j(X))-\frac{1}{\pi}\int_{|z|<1}f_{z_0}(z^n) dz d\xo{z}
    ,
 \end{EA}
where we used a change of variable for the last term.
Below we show that the stochastic domination of \eqref{eq:red_1} by $N^{-1+2d}\|\Delta f\|_{L_1}$ is equivalent to the local circular law for the matrix $X$.
But before that we use Girko's hermitization idea to transform the study of the non-Hermitian matrix $X$ into the study of a family of Hermitian matrices, defined in the next section.

\emph{Hermitization.}
Girko's Hermitization technique relies on the Green's formula for a function with compact support.
\begin{lem}[Green's formula, \cite{SaffToti}]
  Let $f,g:\CC\rightarrow \RR$ be twice continuously differentiable functions and let $R\subset \CC$ be a bounded set with $C^{1}$ boundary $\partial R$.
Let $A(z)$ be Lebesgue measure on $\CC$.
Then
\begin{equation}
  \int_{R} f\Delta g dA(z)-  \int_{R} \Delta f g dA(z)
  =
  -\int_{\partial R}(f \frac{\partial g}{\partial \mathbf{n}}-g\frac{\partial f}{\partial \mathbf{n}})\, ds
\end{equation}
where $\partial/\partial \mathbf{n}$ denotes differentiation in the direction of the inner normal of $R$, and $ds$ indicates integration with respect ot the arc length of $\partial R$.

If we suppose that $f$ has compact support, fix $\tilde{z}\in \CC$ and take $g=-\log |z-\tilde{z}|$ then
\begin{equation}\label{eq:red_0}
    f(\tilde{z})
    =
    \int_{\C} \Delta f(z) \frac{1}{2\pi} \log|z-\tilde{z}|\, dA(z)  
\end{equation}
\end{lem}
Let $\tilde{\mu}_1,\ldots,\tilde{\mu}_{nN}$ denote the eigenvalues of $X$.
Then using \eqref{eq:red_0} we obtain \emph{Girko's hermitization formula}
  \begin{EA}{ll}
    \frac{1}{nN}\sum_{j=1}^{nN} f_{z_0}(\tilde{\mu}^n)
    &=
    \frac{1}{nN}\sum_{j=1}^{nN} \frac{1}{2\pi}\int \Delta \tilde{f}(z)\log |\tilde{\mu}_j -z|dz d\xo{z} 
    =
    \frac{1}{2\pi nN} \int \Delta \tilde{f}(z) \log| \det(X-z)|dz d\xo{z} 
    \\
    &=
    \frac{1}{4\pi nN} \int \Delta \tilde{f}(z) \log| \det(X-z)^* (X-z)|dz d\xo{z} 
    ,
  \end{EA}
where $\tilde{f}(z)=f_{z_0}(z^n)$.
Define $Y_z:=X-z$ and let $\lambda_1<\lambda_2<\ldots<\lambda_{nN}$ be the eigenvalues of $Y_z^*Y_z$.
Then
\begin{equation}
  \frac{1}{nN}\sum_{j=1}^{nN} f_{z_0}(\tilde{\mu}^n_j)
  =
  \frac{1}{4\pi nN} \int \Delta \tilde{f}(z) \sum_{j=1}^{nN}\log \lambda_j(z) dz d\xo{z} 
  .
\end{equation}
We now show how the estimates of \eqref{eq:red_1} can be obtained by studying the eigenvalues $\lambda_j(z)$.

Let $\nu_z$ be a family of the empirical measures on the squared singular values of the matrix $X-z$
\begin{equation}
  \nu_{z,N}(A)
  =
  \frac{1}{nN}\sum_{j=1}^{nN} 1_A(\lambda_j)
  ,\quad
  A\in \Bc(\RR)
  ,
\end{equation}
and let  $m(z,w)$ be the Stieltjes transform of $\nu_{z,N}$
\begin{equation}
  m(z,w)
  =
  \int_{\RR}\frac{1}{x-w}d\nu_{z,N}(x)
  =
  \sum_{j=1}^{nN} \frac{1}{\lambda_j(z)-w}
  .
\end{equation}
The convergence of $m(z,w)$ to a limiting function $m_c(z,w)$, as well as the weak convergence of $\nu_{z,N}$ was shown in \cite[Proposition~1]{NemiNHPO}. Together with \cite[Lemma~11.9]{BaiSilv}, where the authors studies properties of the function $m_c(z,w)$, we have the following result.
\begin{thm}\label{thm:sosh_orou_product}
  There exist a deterministic function $m_c:\CC\times\CC\rightarrow \CC$ and a family of  deterministic measures $\{\nu_z,\, z\in\CC\}$ on $\RR_+$ such that
  \begin{itemize}
  \item[(1)] Almost surely, $m(z,w)$ converges to $m_c(z,w)$ as $N\rightarrow \infty$,
  \item[(2)] Almost surely, $\nu_{z,N}$ converges weakly to $\nu_z$ uniformly in every bounded region of $z$.
  \item[(3)] $\nu_z$ are absolutely continuous measures with density functions $\rho_z$ supported on $(\max \{0,\lambda_-(z)\}, \lambda_+(z))$, where
    \begin{equation}\label{eq:a_and_lambda}
    \lambda_{\pm}(z)=\frac{(\afr\pm 3)^3}{8(\afr\pm 1)}    
    ,\quad 
    \afr:=\sqrt{1+8|z|^2}
    .
    \end{equation}
  \item[(4)] $m_c(z,w)$ is the Stieltjes transform of the measure $\nu_z$, i.e.
    \begin{equation}
      m_c(z,w)
      =
      \int_{\RR} \frac{\rho_z(x)}{x-w} dx
    \end{equation}
  \item[(5)] $m_c(z,w)$ is the solution of the equation
  \begin{equation}\label{eq:mc_equation}
    m_c^{-1}=-w(1+m_c)+|z|^2(1+m_c)^{-1}
  \end{equation}
  that satisfies $\Im m_c(z,w)>0$ if $\Im w>0$.
  \end{itemize}
\end{thm}
Next lemma shows how we can use the properties of $\rho_z$ to reduce \eqref{eq:int_1} to the problem of the rigidity of the singular values $\lambda_j(z)$ around their classical locations.
\begin{lem}(See \cite[Section~5]{BourYauYin})
Let $\gamma_j(z), 1\leq j \leq nN$ be defined by
\begin{equation}
    \int_0^{\gamma_j(z)}\rho_z(x)dx = \frac{j}{nN}
\end{equation}
be classical locations of the eigenvalues of $Y_z^*Y_z$.
Then for any $\varepsilon>0$
\begin{equation}\label{eq:red_2}
  |\sum_{j} \log \gamma_j(z) - nN\int_0^{\infty} (\log x)\rho_z(x)dx|\leq N^{\varepsilon}  
\end{equation}
and
\begin{equation}
  \label{eq:red_3}
  \int_0^{\infty} (\log x)\Delta_z\rho_z(x) dx
  =
  4\cdot 1_{|z|< 1}(z)
.
\end{equation}
\end{lem}
Suppose that 
\begin{equation}\label{eq:red_4}
  |\sum_{j} \log \lambda_j(z) -\sum_{j} \log \gamma_j(z)|
  \leq 
  N^{\varepsilon}
\end{equation}
for any $\varepsilon>0$.
Denote $\tilde{f}(z):=f_{z_0}(z^n)$.
Then
\begin{EA}{ll}
  \frac{1}{nN}\sum_{j=1}^{nN} f_{z_0}(\tilde{\mu}^n)
  &\,=
  \frac{1}{4\pi nN} \int \Delta \tilde{f}(z)\sum_{j=1}^{nN}\log \lambda_j(z) dz d\xo{z} 
  \\
  &\,=
  \frac{1}{4\pi nN} \int \Delta \tilde{f}(z)\sum_{j=1}^{nN}\log \gamma_j(z) dz d\xo{z} +\OO{N^{-1+\varepsilon}\|\Delta \tilde{f}\|_{L_1}}
  \\
  &\,=
  \frac{1}{4\pi} \int \Delta \tilde{f}(z)\int_0^{\infty} (\log x)\rho_{z}(x)dx dz d\xo{z} +\OO{N^{-1+\varepsilon}\|\Delta \tilde{f}\|_{L_1}}
  \\
  &\,=
  \frac{1}{4\pi} \int f_{z_0}(z^n)\int_0^{\infty} (\log x)\Delta\rho_z(x)dx dz d\xo{z} +\OO{N^{-1+\varepsilon}\|\Delta \tilde{f}\|_{L_1}}
  \\
  &\,=
  \frac{1}{\pi} \int f_{z_0}(z^n) 1_{z<1}(z) dz d\xo{z} +\OO{N^{-1+\varepsilon}\|\Delta \tilde{f}\|_{L_1}}
  ,
\end{EA}
where in the first two steps we used \eqref{eq:red_4} and \eqref{eq:red_2}, next we applied integration by parts and finaly we used \eqref{eq:red_3}.
From the properties of Wirtinger derivatives we have the following lemma
\begin{lem}\label{lem:chane_of_variable}
  If $\frac{\partial}{\partial \xo{z}} g=0$, then
  \begin{equation}
    \Delta (f\circ g)
    =
    (\Delta f)\circ g \cdot \frac{\partial g}{\partial z}\frac{\partial \xo{g}}{\partial \xo{z}}
    .
  \end{equation}
\end{lem}
Using the change of variable $\xi= N^{d}(z^n-z_0)$ and the above lemma we get that 
\begin{equation*}
  N^{-1+\varepsilon}\|\Delta \tilde{f}\|_{L_1}
  =
  N^{-1+2d+\varepsilon}\|\Delta f\|_{L_1}
  .
\end{equation*}
Therefore, to prove \eqref{eq:int_1} it is enough to show that
\begin{equation}\label{eq:red_5}
  |\sum_{j} \log \lambda_j(z) -\sum_{j} \log \gamma_j(z)|
  \prec
  1
  .
\end{equation}
In order to obtain the estimate \eqref{eq:red_5}, we proceed as in \cite{BourYauYin}, where the similar estimate was obtained for a matrix with iid entries.
Firstly, we separate a relatively small number of the largest terms, that will be estimated by controlling the smallest singular values $\lambda_1(z)$ and the properties of $\rho_z$.
We shall need the following result proven in \cite{OrouSosh}.
\begin{thm}
For any $A>0$ there exists $B>0$ such that
\begin{equation}
    \Pb[\lambda_1(z)\leq N^{-B}]\leq C N^{-A}
\end{equation}
uniformly in $z$.
\end{thm}
From the properties of $\rho_z$ (see \cite{BourYauYin}, Proposition~3.1) we have that $\gamma_1 \geq C N^{-2}$.
Therefore we conclude that 
\begin{equation}
  |\lambda_1(z)|+|\gamma_1(z)|
  \prec
  1
  .
\end{equation}
Define $\varphi:=(\log N)^{\log \log N}$.
Note that $\varphi$ is asymptotically smaller than $N^{\varepsilon}$ for any $\varepsilon>0$.
Then
\begin{equation}
  |\sum_{j\leq \varphi^C} \log \lambda_j(z)|+|\sum_{j\leq \varphi^C} \log \gamma_j(z)|
  \leq 
  N^{\varepsilon}
  ,
\end{equation}
and it is enough to show that
\begin{equation}\label{eq:reduc_logdiff1}
  \sum_{j> \varphi^C} \log \lambda_j(z)-\log \gamma_j(z)
  \prec
  1
  .
\end{equation}

In \cite{BourYauYin} it was shown that \eqref{eq:reduc_logdiff1} can be obtained from the concentration of the Stieltjes transform stated precisely in the following theorem.
\begin{thm}\label{thm:conc_main}
  There exists $\delta>0$ and $\tilde{Q}$ such that for any $\tau\leq |z|\leq 1-\tau$ or $1+\tau\leq |z|\leq \tau^{-1}$
  \begin{equation}\label{eq:conc_main}
    \sup_{w\in S_{z,\delta,\tilde{Q}}}|m(z,w)-m_c(z,w)|
    \prec
    \frac{1}{N\eta}
    ,
  \end{equation}
where
\begin{equation}\label{eq:s_set}
  S_{z,\delta,\tilde{Q}}=\{E+\sqrt{-1}\eta\,|\,(1+\delta)^{-1}\max\{0,\lambda_-(z)\}\leq E\leq (1+\delta)\lambda_+(z),\varphi^{\tilde{Q}}N^{-1}|m_c|^{-1}\leq \eta \leq 1\}
  .
\end{equation}

\end{thm}
We refer reader to the section~5 in \cite{BourYauYin} for the detailed proof.
The rest of the article is devoted to the proof of Theorem~\ref{thm:conc_main}.


\section{Notations and definitions}
\label{sec:notation}

We start by fixing the notation and giving necessary definitions.
The main argument will follow the general framework proposed by Bourgade, Yau and Yin, therefore we try to keep our notation as close as possible to the notation used in \cite{BourYauYin}.

Throughout the rest of the article $a$ and $b$ will be elements of $\ZZ/n\ZZ$. 

Let $X_a,a\in \ZZ /n \ZZ, $ be independent $N\times N$ matrices, entries of which have zero mean, variance $N^{-1}$ and satisfy condition \eqref{eq:int_2}. Let $X$ be defined by \eqref{eq:linearisation_matrix}.
For  $z\in\CC$ and $w=E+\sqrt{-1}\eta\in\CC_+$ introduce the matrices
\begin{equation*}
  Y_z:=X-z
  ,\quad 
  G(w):=(Y_z^*Y_z-w)^{-1}
  ,\quad
  \Gc(w):=(Y_z Y_z^*-w)^{-1}.
\end{equation*}

We shall consider $nN\times nN$ matrices as consisting of $N\times N$ blocks indexed by $(a,b)$.
We shall use left superscript to specify the submatrix.
For example, for $i,j\in\{1,\ldots,N\}$ 
\begin{equation*}
  \LRI{x}{}{ab}{ij}
  :=
  x_{(\xo{a}-1)N+i,(\xo{b}-1)N+j}
  ,\quad
  \LRI{\Gc}{}{ab}{ij}
  :=
  \Gc_{(\xo{a}-1)N+i,(\xo{b}-1)N+j}.
\end{equation*}
where $\xo{a}\in a\cap\{1,\ldots,n\}$ and $\xo{b}\in b\cap\{1,\ldots,n\}$.

Define also $i(a):=(\xo{a}-1)N+i$ for $1\leq i\leq N$.

We shall use the index $a$ instead of $aa$ for for elements of the diagonal blocks, for example, $\LRI{G}{}{a}{kl}:=G_{k(a),l(a)}$. 

The $i(a)$th rows of matrices $X$ and $Y_z$ will be denoted by $\LRI{x}{}{}{i(a)}$ and $\LRI{y}{}{}{i(a)}$ respectively.
The corresponding columns of these matrices will be denoted by $\LRI{\xb}{}{}{i(a)}$ and $\LRI{\yb}{}{}{i(a)}$.

For $\TT,\UU\subset\{1,\ldots,nN\}$, $Y_z^{(\TT,\UU)}$ will denote a $(nN-|\UU|)\times (nN-|\TT|)$ matrix, obtained from $Y_z$ by deleting the rows with indices in $\UU$ and columns with indices in $\TT$. 
Resolvent matrices, corresponding to these minors, will be denoted by $G^{(\TT,\UU)}$ and $\Gc^{(\TT,\UU)}$, i.e.,
\begin{equation*}
  G^{(\TT,\UU)}(w)
  =
  (Y_z^{(\TT,\UU)*}Y_z^{(\TT,\UU)}-w)^{-1}
  ,\quad
  \Gc^{(\TT,\UU)}(w)
  =
  (Y_z^{(\TT,\UU)}Y_z^{(\TT,\UU)*}-w)^{-1}.
\end{equation*}
Note that we shall keep the indices of the elements of matrices for minors of these matrices. 
More precisely, the entries of $Y_z^{(\TT,\UU)}$ will be indexed by $(\{1,\ldots,nN\}\setminus \UU)\times(\{1,\ldots,nN\}\setminus \TT)$, and the entries of  $G^{(\TT,\UU)}$ and $\Gc^{(\TT,\UU)}$ by $(\{1,\ldots,nN\}\setminus \TT)^2$ and $(\{1,\ldots,nN\}\setminus \UU)^2$ respectively. For $i\in\TT, j\in\UU$ and $k\in\{1,\ldots,nN\}$ we shall define
\begin{equation*}
  G^{(\TT,\UU)}_{ik}
  =G^{(\TT,\UU)}_{ki}
  =0
  ,\quad
  \Gc^{(\TT,\UU)}_{jk}
  =
  \Gc^{(\TT,\UU)}_{kj}
  =0
.
\end{equation*}
Note that all these matrices depend on $z\in\CC$ and $w\in\CC_+$.

For $a\in\ZZ/n\ZZ$ and $\TT,\UU\subset\{1,\ldots,nN\}$, define next
\begin{equation*}
  \LRI{m}{(\TT,\UU)}{a}{G}
  :=
  \frac{1}{N}\sum_{i=1}^{N}\LRI{G}{(\TT,\UU)}{a}{ii}
  ,\quad
  \LRI{m}{(\TT,\UU)}{a}{\Gc}
  :=
  \frac{1}{N}\sum_{i=1}^{N}\LRI{\Gc}{(\TT,\UU)}{a}{ii}
  .
\end{equation*}
If $\TT=\UU=\emptyset$, we shall drop the $(\emptyset,\emptyset)$ superscript and write $\LRI{m}{}{a}{G}$ or $\LRI{m}{}{a}{\Gc}$.
Note, that
\begin{equation*}
  m_G(z,w)
  =
  \frac{1}{n}\sum_{a}\LRI{m}{}{a}{G}(z,w)
  .
\end{equation*}

Now we can introduce
\begin{equation*}
  \Lambda
  :=
  \max_a \{|\LRI{m}{}{a}{G}-m_c|,|\LRI{m}{}{a}{\Gc}-m_c|\}
\end{equation*}
and
\begin{equation*}
  \Psi
  :=
  \sqrt{\frac{\Im m_c +\Lambda}{N\eta}} +\frac{1}{N\eta}
\end{equation*}
where $m_c$ was defined in Theorem~\ref{thm:sosh_orou_product}.

We shall use $C$ and $c$ to denote different constants, that do not depend on $N$, $w$ or $z$.

For $\zeta>0$ we say that an event $\Xi_N$ holds with $\zeta$-high probability, if 
\begin{equation*}
  \Pr{\Xi^{\complement}}
  \leq
  N^{C}e^{-\varphi^{\zeta}}
  ,
\end{equation*}
where 
\begin{equation}
  \label{eq:varphi}
  \varphi
  :=
  (\log N)^{\log \log N}
\end{equation}
and $C>0$.

For $A,B>0$ we shall write  $A\sim_c B$ or simply $A\sim B$ if there is $c>0$ such that 
\begin{equation}
  \label{eq:equivalence_def}
   c^{-1}A\leq B\leq c A.
\end{equation}


\section{Tools and Methods}
\label{sec:tools}
This section collects some basic and classical tools which will be relevant towards the main result. 
Note that in Lemmas \ref{lem:bord_guion_minor_estimate}, \ref{lem:im_of_resolvent_diag_and_product}, \ref{lem:minor_differences} and \ref{lem:derivative_estimates} we deal with objects introduced in Section~\ref{sec:notation}, while in lemmas \ref{lem:mc_properties_0}, \ref{lem:mc_properties_1} and \ref{lem:mc_properties_2} we recall properties of the function $m_c$, which was introduced in Theorem~\ref{thm:sosh_orou_product}.

\subsection{Linear Algebra}
\label{sec:linear_algebra}

\begin{lem}\emph{(Schur complement formula, \cite[Section~0.7.3]{HornJohn})}
  Let $A$ be an invertible matrix and let $B$ be its inverse. Divide the matrices $A$ and $B$ into blocks
  \begin{equation*}
    A=
    \begin{pmatrix}
      A_{11}&A_{12}
      \\
      A_{21}&A_{22}
    \end{pmatrix}
    ,\quad
    B=
    \begin{pmatrix}
      B_{11}&B_{12}
      \\
      B_{21}&B_{22}
    \end{pmatrix},
  \end{equation*}
so that the blocks with the same index have the same size, and the blocks on the diagonal are square submatrices.

If $A_{22}$ is invertible, then
  \begin{equation}\label{eq:schur_complement}
    \left[B_{11}\right]^{-1}
    =
    A_{11}-A_{12}\left[A_{22}\right]^{-1}A_{21}
    .
  \end{equation}
\end{lem}

\begin{lem}
  Let $A$ be a square matrix and let $w$ be a complex number. If $A^*A-\omega$ is invertible, then
    \begin{equation*}
      A(A^*A-w)^{-1}A^*
      =
      I+w(AA^* -w)^{-1}
      .
    \end{equation*}
\end{lem}

\begin{proof}
  Follows from Woodbury matrix identity (see \cite[Section~0.7.4]{HornJohn}).
\end{proof}

\begin{lem}\emph{\cite[proof of Lemma~C.3]{BordGuio}}\label{lem:bord_guion_minor_estimate}
  Let $\TT,\UU,\Kc\subset\{1,\ldots,nN\}$. Then for any $i\notin \TT$
  \begin{equation}\label{eq:bord_guion_minor_estimate}
    |\sum_{k\in
      \Kc} G^{(\TT i,\UU)}_{kk}-G^{(\TT ,\UU)}_{kk}|
    \leq
    \frac{4}{\eta}
    ,\quad
    |\sum_{k\in \Kc} G^{(\UU,\TT i)}_{kk}-G^{(\UU,\TT)}_{kk}|
    \leq
    \frac{4}{\eta}
    .
  \end{equation}
The same is true for $\Gc$.
\end{lem}

\begin{lem}\label{lem:im_of_resolvent_diag_and_product}
For $i\in\{1,\ldots,nN\}$
  \begin{equation}\label{eq:im_of_resolvent_diag_and_product}
    \sum_{k=1}^{nN}|G_{ki}|^2
    =
    \frac{\Im G_{ii}}{\eta}
    .
  \end{equation}
\end{lem}
\begin{proof}
Let $\lambda_1,\ldots,\lambda_{nN}$ be eigenvalues of the matrix $Y_z^* Y_z$. Then $G=U^* D U$, where $U$ is unitary and $D=\mathrm{diag}((\lambda_1-w)^{-1},\ldots,(\lambda_{nN}-w)^{-1})$. Note that
  \begin{equation*}
    \sum_{k=1}^{nN}|G_{ki}|^2
    =
    \left[G^* G\right]_{ii}
    =
    \left[U^*D^* D U\right]_{ii}
    .
  \end{equation*}
The lemma now follows from the relation
\begin{equation*}
  \frac{1}{\eta}\Im \frac{1}{\lambda_i-w}
  =
  \frac{1}{|\lambda_i-w|^2}
  ,\quad
  i\in\{1,\ldots,nN\}
  .
\end{equation*}
\end{proof}

\begin{lem}\emph{(\cite[Lemma~6.3]{BourYauYin})}\label{lem:minor_differences}
  Let $\TT,\UU\subset\llbracket 1,nN\rrbracket$. If $i,j\notin \TT\cup \{k\}$ then
  \begin{equation}
    \label{eq:minor_differences_1}
    G_{ij}^{(\TT,\UU)}-G_{ij}^{(\TT k,\UU)}
    =
    \frac{G_{ik}^{(\TT,\UU)} G_{kj}^{(\TT,\UU)}}{G_{kk}^{(\TT,\UU)}}
    ,\quad
    \Gc_{ij}^{(\UU,\TT)}-\Gc_{ij}^{(\UU ,\TT k)}
    =
    \frac{\Gc_{ik}^{(\UU,\TT)} \Gc_{kj}^{(\UU,\TT)}}{\Gc_{kk}^{(\UU,\TT)}}
    ,
  \end{equation}
  \begin{equation}
    \label{eq:minor_differences_2}
    G^{(\TT,\UU)}-G^{(\TT,\UU k)}
    =
    -\frac{\left(G^{(\TT,\UU k)}y_{k}^* \right) \left(y_k G^{(\TT,\UU k)}\right)}{1+y_k G^{(\TT,\UU k)} y_k^*}
    ,\quad
    \Gc^{(\TT,\UU)}-\Gc^{(\TT k,\UU)}
    =
    -\frac{\left(\Gc^{(\TT k,\UU )} \yb_{k} \right) \left(\yb_k^* \Gc^{(\TT k,\UU)} \right)}{1+\yb_k^* \Gc^{(\TT k,\UU)} \yb_k}
    .
  \end{equation}
\end{lem}

\begin{lem}\label{lem:derivative_estimates}
Let $w=E+\sqrt{-1}\eta\in\CC_+$. 
Then for any $i\in\llbracket 1,nN\rrbracket$
\begin{EA}{c}
  \left|\frac{\partial G_{ii}}{\partial E}\right|
  +
  \left|\frac{\partial G_{ii}}{\partial \eta}\right|
  +
  \left|\frac{\partial \Gc_{ii}}{\partial E}\right|
  +
  \left|\frac{\partial \Gc_{ii}}{\partial \eta}\right|
  =
  \OO{\frac{1}{\eta^2}}.
\end{EA}
\end{lem}
\begin{proof}
  With the notation used in the proof of Lemma~\ref{lem:im_of_resolvent_diag_and_product}, for any $i\in\llbracket 1,nN\rrbracket$
  \begin{equation*}
    G_{ii}
    =
    \sum_{j=1}^{nN} |u_{ij}|^2 \frac{1}{\lambda_{j}-E-i\eta}
    ,
  \end{equation*}
where $u_{ij}$ are the entries or the unitary matrix $U$.
Therefore, the bound for $G_{ii}$ follows from the fact that
\begin{equation*}
  \left|\frac{\partial}{\partial \eta}\left(\frac{1}{\lambda_{j}-E-\sqrt{-1}\eta}\right)\right|+\left|\frac{\partial}{\partial E}\left(\frac{1}{\lambda_{j}-E-\sqrt{-1}\eta}\right)\right|
  =
  \OO{\frac{1}{\eta^2}}
  .
\end{equation*}
The proof for $\Gc_{ii}$ is similar.
\end{proof}

\subsection{Properties of $m_c$}

\begin{lem}\emph{(\cite[Lemma~4.1]{BourYauYin})}\label{lem:mc_properties_0}
There exists $\tau_0>0$ such that for any $\tau\leq \tau_0$ if $|z|\leq 1-\tau$ and $|w|\leq \tau^{-1}$ then the following properties concerning $m_c$ hold:
\begin{itemize}
\item[(1)] 
 If $E\geq \lambda_+$ and $|w-\lambda_+|\geq \tau$, then
  \begin{equation}
    |\Re m_c|\sim 1
    ,\quad
    -\frac{1}{2}\leq \Re m_c <0
    ,\quad
    \Im m_c \sim \eta
    ,
  \end{equation}
\item[(2)]
  If $|w-\lambda_+|\leq \tau$, then
  \begin{equation}\label{eq:prop_mc_lambda_1}
    m_c(z,w)
    =
    -\frac{2}{3+\afr}+\sqrt{\frac{8(1+\afr)^{3}}{\afr(3+\afr)^{5}}}(w-\lambda_+)^{1/2}+\OO{\lambda_+-w}
    ,
  \end{equation}
  and
  \begin{equation}
    \Im m_c
    \sim
    \left\{
      \begin{array}{ll}
        \frac{\eta}{\sqrt{|E-\lambda_+|}}&\mbox{ if } |E-\lambda_+|\geq \eta \mbox{ and } E\geq \lambda_+
                                           ,\\
        \sqrt{\eta} & \mbox{ if } |E-\lambda_+|\leq \eta \mbox{ or } E\leq \lambda_+
                      ,
      \end{array}
    \right.
  \end{equation}
where $\afr$ and $\lambda_{\pm}$ were defined in \eqref{eq:a_and_lambda}.
\item[(3)]
  If $|w|\leq \tau$, then
  \begin{equation}\label{eq:prop_mc_zero}
    m_c(z,w)
    =
    \sqrt{-1}\frac{1-|z|^2}{\sqrt{w}}+\frac{1-2|z|^2}{2|z|^2-2} +\OO{\sqrt{w}}
  \end{equation}
\item[(4)]
  If $|w|\geq \tau$, $|w-\lambda_+|\geq \tau$ and $E\leq \lambda_+$, then
  \begin{equation}
    |m_c|\sim 1
    ,\quad
    \Im m_c \sim 1
  \end{equation}
\end{itemize}
\end{lem}

\begin{lem}\emph{(\cite[Lemma~4.2]{BourYauYin})}\label{lem:mc_properties_1}
There exists $\tau_0>0$ such that for any $\tau\leq \tau_0$ if $|z|\geq 1+\tau$ and $|w|\leq \tau^{-1}$ then the following properties concerning $m_c$ hold
\begin{itemize}
\item[(1)] 
 If $E\geq \lambda_+$ and $|w-\lambda_+|\geq \tau$, then
  \begin{equation}
    |\Re m_c|\sim 1
    ,\quad
    -\frac{1}{2}\leq \Re m_c <0
    ,\quad
    \Im m_c \sim \eta
    ,
  \end{equation}
\item[(2)] 
 If $E\leq \lambda_-$ and $|w-\lambda_-|\geq \tau$, then
  \begin{equation}
    |\Re m_c|\sim 1
    ,\quad
    0\leq \Re m_c 
    ,\quad
    \Im m_c \sim \eta
    ,
  \end{equation}
\item[(3)]
  If $|w-\lambda_{\pm}|\leq \tau$, then
  \begin{equation}\label{eq:prop_mc_lambda_2}
    m_c(z,w)
    =
    -\frac{2}{3\pm\afr}+\sqrt{\frac{8(\afr\pm 1)^{3}}{\pm\afr(\afr\pm 3)^{5}}}(w-\lambda_{\pm})^{1/2}+\OO{\lambda_{\pm}-w}
    ,
  \end{equation}
  and
  \begin{equation}
    \Im m_c
    \sim
    \left\{
      \begin{array}{ll}
        \frac{\eta}{\sqrt{|E-\lambda_{\pm}|}}&\mbox{ if } |E-\lambda_{\pm}|\geq \eta \mbox{ and } E\notin [\lambda_-,\lambda_+]
                                           ,\\
        \sqrt{\eta} & \mbox{ if } |E-\lambda_{\pm}|\leq \eta \mbox{ or } E\in [\lambda_-,\lambda_+]
                      .
      \end{array}
    \right.
  \end{equation}
\item[(4)]
  If $|w-\lambda_{\pm}|\geq \tau$ and $\lambda_-\leq E\leq \lambda_+$, then
  \begin{equation}
    |m_c|\sim 1
    ,\quad
    \Im m_c \sim 1
  \end{equation}
\end{itemize}
\end{lem}

\begin{lem}\emph{(\cite[Lemma~4.3]{BourYauYin})}\label{lem:mc_properties_2}
 There exists $\tau_0>0$ such that for any $\tau\leq \tau_0$ if either the conditions $|z|\leq 1-\tau$ and $|w|\leq \tau^{-1}$ hold or the conditions $|z|\geq 1+\tau$, $|w|\leq \tau^{-1}$, $\Re w \geq \lambda_-/5$ hold, then we have the following bounds
 \begin{EA}{ll}\EAy\label{eq:mc_equivalences}
   |m_c|\sim |1+m_c| \sim |w|^{-1/2}
   \\\EAy\label{eq:immc_bound}
   |\Im \frac{1}{w(1+m_c)}|\leq C \Im m_c
 \end{EA}
\end{lem}



\subsection{McDiarmid's Concentration Inequality}

\begin{thm}\emph{(\cite{Mcdi})}
Let $U=(u_1,\ldots,u_N)$ be a family of independent random variables taking values in the set $A$.
Suppose that the real-valued function $f:A^N\rightarrow \RR$ satisfies
\begin{equation}\label{eq:mcdiarmid_condition}
  |f(u)-f(u')|\leq c_k
\end{equation}
if the vectors $u$ and $u'$ differs only in $k$th coordinate.
Then for any $t\geq 0$
\begin{equation*}
  \Pr{|f(U)-\Me{f(U)}|\geq t}
  \leq
  2 e^{-2t^2/\sum c_k^2}
  .
\end{equation*}
\end{thm}

\subsection{Abstract Decoupling Lemma}

\begin{thm} \emph{(Abstract decoupling lemma, \cite[Lemma~7.3]{PillYin})}\label{thm:abs_dec}
    Let $\mathcal{I}$ be a finite set which may depend on $N$ and let $\mathcal{I}_i\subset\mathcal{I}, 1\leq i\leq N$.  Let $\{x_{\alpha},\alpha\in\mathcal{I}\}$ be a collection of independent random variables and $S_1,\ldots,S_N$ be random variables which are functions of $\{x_{\alpha},\alpha\in\mathcal{I}\}$.
  Let $\mathbb{E}_i$ denote the expectation value operator with respect to $\{x_{\alpha},\alpha\in\mathcal{I}_i\}$.
  Define the commuting projection operators
  \begin{equation*}
    \QQ_i=1-\mathbb{E}_i,P_i=\mathbb{E}_i, P_i^2=P_i, \QQ_i^2=\QQ_i, [\QQ_i,P_j]=[P_i,P_j]=[\QQ_i,\QQ_j]=0
  \end{equation*}
  and for $\AA\subset\{1,2,\ldots,N\}$,
  \begin{equation*}
    \QQ_{\AA}:=\prod_{i\in \AA}\QQ_i, P_{\AA}:=\prod_{i\in \AA}P_i
  \end{equation*}
  We use the notation
  \begin{equation*}
    [\QQ S]=\frac{1}{N}\sum_{i=1}^{N}\QQ_iS_i
  \end{equation*}

  Let $\Xi$ be an event and $p$ an even integer, which may depend on $N$. Suppose the following assumptions hold with some constants $C_0, c_0>0$.
  \begin{itemize}
  \item[(i)]  There exist deterministic positive numbers $\mathcal{X}<1$ and $\mathcal{Y}$ such that for any set $\AA\subset\{1,2,\ldots,N\}$ with $i\in \AA$ and $|\AA|\leq p, \QQ_{\AA}S_i$ in $\Xi$ can be written as the sum of two new random variables:
  \begin{equation}\label{eq:abs_dec_cond_i}
      1(\Xi)(\QQ_{\AA}S_i)=S_{i,\AA}+1(\Xi)\QQ_{\AA} 1(\Xi^{\complement})\tilde{S}_{i,\AA}
  \end{equation}
  and
  \begin{equation*}
    |S_{i,\AA}|\leq \mathcal{Y}(C_0\mathcal{X}|\AA|)^{|\AA|}, |\tilde{S}_{i,A}|\leq \mathcal{Y}N^{C_0|\AA|}
    ;
  \end{equation*}
  \item[(ii)] 
    \begin{equation}\label{eq:abs_dec_cond_ii}
      \max_{i}|S_i|\leq\mathcal{Y}N^{C_0}
      ;
    \end{equation}
  \item[(iii)] 
    \begin{equation}\label{eq:abs_dec_cond_iii}
      \Pb[\Xi^c]\leq e^{-c_0(\log N)^{3/2}p}
      .
    \end{equation}
  \end{itemize}
  Then, under the assumptions (i), (ii), (iii) above, we have
  \begin{equation*}
    \E[\QQ S]^p\leq (Cp)^{4p}[\mathcal{X}^{2}+N^{-1}]^{p}\mathcal{Y}^p
  \end{equation*}
  for some $C>0$ and any sufficiently large $N$.
\end{thm}


\section{Concentration of the Stieltjes transform}


\subsection{System of ``self-consistent equations''}
The aim of this section is to prove Theorem~\ref{thm:sce}. 
We begin with three independent lemmas.
\begin{lem}\label{lem:sce_schur}
  For any $\TT,\UU\subset \{1,\ldots,nN\}$ and $i,j\in\{1,\ldots,nN\}\setminus\TT,i\neq j$, we have
  \begin{EA}{c}\EAy\label{eq:sce_schur_1}
    \frac{1}{G_{ii}^{(\TT,\UU)}}
    =
    -w(1+\yb_i^*\Gc^{(\TT i,\UU)}\yb_i)
    ,\quad
    G^{(\TT,\UU)}_{ij}
    =    
    -wG^{(\TT,\UU)}_{ii}G^{(\TT i,\UU)}_{jj}\left(\yb_i^*\Gc^{(\TT ij,\UU)}\yb_j\right)
    ,\\ \EAy \label{eq:sce_schur_2}
    \frac{1}{\Gc_{ii}^{(\UU,\TT)}}
    =
    -w(1+y_iG^{(\UU,\TT i)}y_i^*)   
    ,\quad
    \Gc^{(\UU,\TT)}_{ij}
    =
    -w\Gc^{(\UU,\TT)}_{ii}\Gc^{(\UU ,\TT i)}_{jj}\left(y_iG^{(\UU,\TT ij)}y_j^*\right)
    .
  \end{EA}
\end{lem}
\begin{proof}
  See \cite[Lemma~6.5]{BourYauYin}.
\end{proof}

Define  subsets of $\CC$
\begin{EA}{l}
  \Zfr_{\tau}
  :=
  \{|z|\leq 1-\tau\}\cup\{1+\tau \leq |z|\leq \tau^{-1}\}
  ,\\
  S_{z,\delta,0}
  :=
  \{w\in\CC_+\,:\,  w=E+\sqrt{-1}\eta,\, \max \{0,(1+\delta)^{-1}\lambda_-(z)\}\leq E\leq(1+\delta) \lambda_+(z),\, 0\leq\eta\leq 1\}
  .
\end{EA}

Next two lemmas are technical results which provide means to make manipulations with functions that approximate $m_c(z,w)$ easier.
Similar results but for different values of $z$ and $w$ were proven in \cite[Lemmas 13 and 14]{NemiNHPO}.
By tracking the changes in the behaviour of $m_c$ in different regions of $z$ and $w$ (see lemmas \ref{lem:mc_properties_0} and \ref{lem:mc_properties_1}), the arguments can be adapted without difficulties to the new setting, therefore we state these lemmas without proof.
\begin{lem}\label{lem:prop_of_mc_close_functions}
There exist $\alpha>0$ small enough and $C>0$, such that for any $h_i:\CC\times\CC_+\rightarrow \CC,i\in\{1,2\}$, for all $z\in\Zfr_{\tau}$ and $w\in S_{z,1,0}$ if
  \begin{equation*}
    \max_{i\in\{1,2\}}|h_i(z,w)-m_c(z,w)|\leq 2\alpha |m_c(z,w)|
  \end{equation*}
holds, then
  \begin{EA}{rll}\EAy\label{eq:prop_of_mc_close_functions_1}
   (i) & \quad &\left| 1+h_1(z,w)\right| 
   \sim_C
   |h_1(z,w)|
   \sim_C
   |w|^{-1/2}
   ,\\ \EAy\label{eq:prop_of_mc_close_functions_2}
   (ii) & \quad &\left|\Im \frac{1}{w(1+h_1(z,w))}\right| 
   \leq
   C\left(\Im m_c(z,w)+|h_1(z,w)-m_c(z,w)|\right)
   , \\ \EAy\label{eq:prop_of_mc_close_functions_3}
   (iii) & \quad &\left|(1+h_1(z,w))-\frac{|z|^2}{w(1+h_2(z,w))}+\frac{1}{wm_c(z,w)}\right|
   \leq
   C\max_i|h_i(z,w)-m_c(z,w)|.
  \end{EA}
\end{lem}

\begin{lem}\label{lem:lde}
Let $\zeta>0$. 
Then there exists $Q_{\zeta}>0$ such that for all sufficiently large $N$, for any $\TT,\UU\subset \{1,\ldots,nN\}$, for any $a\in\ZZ/n\ZZ$ and $\{i,j\}\subset\{1,\ldots,N\}$ such that $\{i(a),j(a)\}\subset\TT$ ($i=j$ is allowed), for $z\in\Zfr_{\tau}$ and $w\in S_{z,1,0}$  with $\zeta$-high probability
  \begin{equation}\label{eq:lde_1}
    (1-\EE_{\yb_{i(a)}\yb_{j(a)}})\left[\sum_{k,l=1}^{nN} \xo{x}_{ki(a)} \Gc^{(\TT,\UU)}_{kl} x_{lj(a)}  \right]
    =
    \OO{\varphi^{Q_{\zeta}}\left(\Psi+\frac{|\TT|+|\UU|}{N\eta}\right)}
    ,
  \end{equation}
and if $i(a)\notin \UU$, then
\begin{equation}\label{eq:lde_2}
    (1-\EE_{\yb_{i(a)}})\left[\sum_{k}^{nN} \xo{x}_{ki(a)} \Gc^{(\TT,\UU)}_{ki(a)}  \right]
    =
    \OO{\varphi^{Q_{\zeta}}\sqrt{\frac{\Im \LRI{\Gc}{(\TT,\UU)}{a}{ii}}{N\eta}}}
    ,\quad
    (1-\EE_{\yb_{i(a)}})\left[\sum_{k}^{nN} \Gc^{(\TT,\UU)}_{i(a)k} {x}_{i(a)k} \right]
    =
    \OO{\varphi^{Q_{\zeta}}\sqrt{\frac{\Im \LRI{\Gc}{(\TT,\UU)}{a}{ii}}{N\eta}}}
    .
\end{equation}
\end{lem}
The above result is valid if we take the matrix $G^{(\UU,\TT)}$ and rows $y_{i(a)}$ instead of $\Gc^{(\TT,\UU)}$ and $\yb_{i(a)}$.
From now on we fix $\alpha$ as in Lemma~\ref{lem:prop_of_mc_close_functions} and $Q_{\zeta}$ as in Lemma~\ref{lem:lde}.

To state and prove our next result we shall need some additional notation.

For $a\in\ZZ/n\ZZ$, $i\in \llbracket 1,N\rrbracket$ and $\TT\subset\llbracket 1,nN\rrbracket$ we define
\begin{equation*}
  \LRI{Z}{(\TT)}{a}{i}
  :=
  \Mes{y_{i(a)}}{y_{i(a)} G^{(\TT,\emptyset)} y_{i(a)}^*}
  ,\quad
  \LRI{\Zc}{(\TT)}{a}{i}
  :=
  \Mes{\yb_{i(a)}}{\yb_{i(a)}^* \Gc^{(\emptyset,\TT)} \yb_{i(a)}}
  ,
\end{equation*}
and we shall suppress the right superscript if $\TT=\emptyset$.

For any $t>0$ define an $N$-dependent set
\begin{equation*}
  S_{z,\delta,t}
  :=
  \{w=E+\sqrt{-1}\eta\, | \, (1+\delta)^{-1} \max\{0,\lambda_-(z)\}\leq E\leq (1+\delta)\lambda_+(z), \frac{\varphi^{t}}{N|m_c|}\leq \eta\leq1\}
  =
  S_{z,\delta,0}\cap \left\{\eta\geq \frac{\varphi^t}{N|m_c|}\right\}
  .
\end{equation*}

We are now in position to prove the main result of this section.
Although the proof of this theorem mimics the argument used in \cite[Theorem~6]{NemiNHPO}, for reader's convenience we provide here a complete proof.

\begin{thm}\label{thm:sce}
  For any $\zeta>0$ there exists $\tilde{Q}_{\zeta}>0$ such that the following implication is true for all $z\in\Zfr_{\tau}$ and $w\in S_{z,1,\tilde{Q}_{\zeta}}$:
\newline
if
  \begin{equation}\label{eq:alpha_condition}
    \Lambda(z,w)
    \leq
    \alpha |m_c(z,w)|
  \end{equation}
holds with $\zeta$-high probability, then
  \begin{EA}{c}\EAy\label{eq:sce_1}
    \LRI{G}{(\emptyset,i(a))}{a}{ii}
    =
    \left[-w(1+\LRI{m}{}{a-1}{\Gc})\right]^{-1} + \OO{\varphi^{Q_{\zeta}}\Psi}
    ,\quad
    a\in\ZZ/n\ZZ,\,1\leq i\leq N
    ,\\\EAy\label{eq:sce_2}
    \LRI{\Gc}{(i(a),\emptyset)}{a}{ii}
    =
    \left[-w(1+\LRI{m}{}{a+1}{G})\right]^{-1} + \OO{\varphi^{Q_{\zeta}}\Psi}
    ,\quad
    a\in\ZZ/n\ZZ,\,1\leq i\leq N
    ,\\\EAy\label{eq:sce_3}
    \LRI{G}{}{a}{ii}
    =
    \left[-w(1+\LRI{m}{}{a-1}{\Gc})+\frac{|z|^2}{1+\LRI{m}{}{a+1}{G}}\right]^{-1} + \OO{\varphi^{2Q_{\zeta}}\Psi}
    ,\quad
    a\in\ZZ/n\ZZ,\,1\leq i\leq N
    ,\\\EAy\label{eq:sce_4}
    \LRI{\Gc}{}{a}{ii}
    =
    \left[-w(1+\LRI{m}{}{a+1}{G})+\frac{|z|^2}{1+\LRI{m}{}{a-1}{\Gc}}\right]^{-1} + \OO{\varphi^{2Q_{\zeta}}\Psi}
    ,\quad
    a\in\ZZ/n\ZZ,\,1\leq i\leq N
    ,\\\EAy\label{eq:sce_mg}
    \frac{1}{w\LRI{m}{}{a}{G}}+(1+\LRI{m}{}{a-1}{\Gc})-\frac{|z|^2}{w(1+\LRI{m}{}{a+1}{G})} 
    =
    \OO{\varphi^{2Q_{\zeta}}\Psi}
    ,\quad 
    a\in\ZZ/n\ZZ
    ,\\\EAy\label{eq:sce_mgc}
    \frac{1}{w\LRI{m}{}{a}{\Gc}}+(1+\LRI{m}{}{a+1}{G})-\frac{|z|^2}{w(1+\LRI{m}{}{a-1}{\Gc})} 
    =
    \OO{\varphi^{2Q_{\zeta}}\Psi}
    ,\quad
    a\in\ZZ/n\ZZ
    ,
  \end{EA}
hold with $\zeta$-high probability.
\end{thm}
\begin{proof}
We begin with equation \eqref{eq:sce_2}.
Using \eqref{eq:sce_schur_2} and taking the expectation with respect to $\yb_{i(a)}$
\begin{EA}{ll}
    \LRI{\Gc}{(i(a),\emptyset)}{a}{ii}
    &=
    \frac{1}{-w\left(1+\LRI{m}{(i(a),i(a))}{a+1}{G}+\LRI{Z}{(i(a))}{a}{i}\right)}
    \\
    &=
    \frac{1}{-w\left(1+\LRI{m}{}{a+1}{G}+(\LRI{m}{(i(a),i(a))}{a+1}{G}-\LRI{m}{}{a+1}{G})+\LRI{Z}{(i(a))}{a}{i}\right)}
    \\
    &=
    \frac{1}{-w\left(1+\LRI{m}{}{a+1}{G}\right)}+\frac{(\LRI{m}{(i(a),i(a))}{a+1}{G}-\LRI{m}{}{a+1}{G})+\LRI{Z}{(i(a))}{a}{i}}{w(1+\LRI{m}{}{a+1}{G})(1+\LRI{m}{(i(a),i(a))}{a+1}{G}+\LRI{Z}{(i(a))}{a}{i})}
    .
\end{EA}
The $i(a)$th row and column of $\LRI{G}{(i(a),i(a))}{}{}$ are equal to zero by definition. Therefore
\begin{equation*}
  y_{i(a)} \LRI{G}{(i(a),i(a))}{}{} y_{i(a)}^*
  =
  \sum_{k,l=1}^{N} x_{i(a)k(a+1)} \LRI{G}{(i(a),i(a))}{a+1}{kl}\xo{x}_{i(a)l(a+1}
\end{equation*}
and from \eqref{eq:lde_1} we have that $\left|\LRI{Z}{(i(a))}{a}{i}\right|=\OO{\varphi^{Q_{\zeta}}\Psi}$ for some $Q_{\zeta}>0$.

Suppose that $\tilde{Q}_{\zeta}>6 Q_{\zeta}$.
Then
\begin{equation*}
  \varphi^{2Q_{\zeta}}\Psi
  \leq
  \varphi^{2Q_{\zeta}}(\sqrt{\frac{(1+\alpha)|m_c|}{N\eta}}+\frac{1}{N\eta})
  \leq
  \varphi^{2Q_{\zeta}}|m_c|(\sqrt{\frac{(1+\alpha)}{|m_c|N\eta}}+\frac{1}{|m_c|N\eta})
  \leq
  \varphi^{-Q_{\zeta}}|m_c|.
\end{equation*}
Recall that by \eqref{eq:bord_guion_minor_estimate}
\begin{equation*}
  |\LRI{m}{(i(a),i(a))}{a+1}{G}-\LRI{m}{}{a+1}{G}|
  \leq
  \frac{8}{N\eta}.
\end{equation*}
If $N$ is big enough, then
\begin{equation*}
  |\LRI{m}{(i(a),i(a))}{a+1}{G}-m_c|
  \leq
  2 \alpha |m_c|
\end{equation*}
and from \eqref{eq:prop_of_mc_close_functions_1} we get \eqref{eq:sce_2}.

We now apply \eqref{eq:sce_schur_2} to $\left[\LRI{G}{}{a}{ii}\right]^{-1}$, take expectation with respect to the column $\yb_{i(a)}$ and use \eqref{eq:sce_2}
\begin{EA}{ll}
    \frac{1}{w\LRI{G}{}{a}{ii}}
    &=
    -\left(1+\LRI{m}{(i(a),\emptyset)}{a-1}{\Gc}+\LRI{\Zc}{}{a}{i} +|z|^2\LRI{\Gc}{(i(a),\emptyset)}{a}{ii}\right)
    \\ \EAy \label{eq:sce_5}
    &=
    -(1+\LRI{m}{}{a-1}{\Gc})-\LRI{\Zc}{}{a}{i} +\frac{|z|^2}{w(1+\LRI{m}{}{a+1}{G})} +\OO{\varphi^{Q_{\zeta}}\Psi}
    .
\end{EA}
We estimate $\LRI{\Zc}{}{a}{i}$ using Lemma~\ref{lem:lde} and \eqref{eq:sce_2} as
\begin{EA}{ll}
  \LRI{\Zc}{}{a}{i}
  &=
  (1-\EE_{\yb_{i(a)}})\left[\sum_{k,l=1}^{N} \xo{x}_{k(a-1)i(a)}\LRI{\Gc}{(i(a),\emptyset)}{a-1}{kl} x_{l(a-1)i(a)} 
    -\xo{z}\sum_{l=1}^{N} \LRI{\Gc}{(i(a),\emptyset)}{}{i(a)l(a-1)} x_{l(a-1)i(a)}
    -z\sum_{k=1}^{N} \xo{x}_{k(a-1)i(a)}\LRI{\Gc}{(i(a),\emptyset)}{}{k(a-1)i(a)} 
  \right]
  \\
  &=
  \OO{\varphi^{Q_{\zeta}}\Psi}+\OO{\varphi^{Q_{\zeta}}\sqrt{\frac{\Im \LRI{\Gc}{(i(a),\emptyset)}{a}{ii}}{N\eta}}}
  \\
  &=
  \OO{\varphi^{Q_{\zeta}}\Psi}+\OO{\varphi^{Q_{\zeta}}\sqrt{\frac{1}{N\eta}\left(\Im \frac{1}{w(1+\LRI{m}{}{a+1}{G})}+\OO{\varphi^{Q_{\zeta}}\Psi}\right)}}
  .
\end{EA}
Then by \eqref{eq:prop_of_mc_close_functions_2}
\begin{equation*}
  \sqrt{\frac{1}{N\eta}\left(\Im \frac{1}{w(1+\LRI{m}{}{a+1}{G})}\right)}
  =
  \OO{\sqrt{\frac{\Im m_c + \Lambda}{N\eta}}}.
\end{equation*}
We conclude that
  \begin{equation*}
    \LRI{\Zc}{}{a}{i}
    =
    \OO{\varphi^{2Q_{\zeta}}\Psi}
  \end{equation*}
and thus 
  \begin{equation*}
    \frac{1}{w\LRI{G}{}{a}{ii}}
    =
    -(1+\LRI{m}{}{a-1}{\Gc})+\frac{|z|^2}{w(1+\LRI{m}{}{a+1}{G})}+\OO{\varphi^{2Q_{\zeta}}\Psi}.
  \end{equation*}
Now by \eqref{eq:prop_of_mc_close_functions_3}
\begin{equation*}
\left[-(1+\LRI{m}{}{a-1}{\Gc})+\frac{|z|^2}{w(1+\LRI{m}{}{a+1}{G})}+\OO{\varphi^{2Q_{\zeta}}\Psi}\right]^{-1}
=
\left[-(1+\LRI{m}{}{a-1}{\Gc})+\frac{|z|^2}{w(1+\LRI{m}{}{a+1}{G})}\right]^{-1}+w\OO{\varphi^{2Q_{\zeta}}\Psi}
\end{equation*}
and the equation \eqref{eq:sce_3} is proved.

If we sum the left- and right-hand sides of \eqref{eq:sce_3} over $i\in\{1,\ldots,N\}$ and divide by $N$, we get
\begin{equation*}
  \LRI{m}{}{a}{G}
  =
  \left[-w(1+\LRI{m}{}{a-1}{\Gc}+\frac{|z|^2}{1+\LRI{m}{}{a+1}{G}})\right]^{-1}+\OO{\varphi^{2Q_{\zeta}}\Psi}
  .
\end{equation*}
Using again \eqref{eq:prop_of_mc_close_functions_3} we have
  \begin{equation*}
    \frac{1}{w\LRI{m}{}{a}{G}}
    =
    -(1+\LRI{m}{}{a-1}{\Gc})+\frac{|z|^2}{w(1+\LRI{m}{}{a+1}{G})}+\OO{\varphi^{2Q_{\zeta}}\Psi}
    .
  \end{equation*}

Equations \eqref{eq:sce_1}, \eqref{eq:sce_4} and \eqref{eq:sce_mgc} can be proved in the same way.
Theorem~\ref{thm:sce} is established.
\end{proof}


\subsection{Weak concentration}

In this section we study the stability properties of the solutions of the system \eqref{eq:sce_mg}-\eqref{eq:sce_mgc} and obtain the initial estimate for $\Lambda$.
Although the  derivation of the self-consistent equations \eqref{eq:sce_mg}-\eqref{eq:sce_mgc} is similar to the case of one matrix considered in \cite{BourYauYin} or to the case of products of matrices but on the different sets of $z$ and $w$ (as in \cite{NemiNHPO}), the analysis of this system in our setting requires much more technical efforts.
This is due to the fact, that the matrix $\Gamma$ defined below, that corresponds to the linearization of the system \eqref{eq:sce_mg}-\eqref{eq:sce_mgc}, is singular at $\lambda_{\pm}$.
Therefore we need to study carefully the behaviour of $\|\Gamma\|_{\infty}$ around these critical point.

As in \cite{NemiNHPO} we start by linearizing the system \eqref{eq:sce_mg}-\eqref{eq:sce_mgc}.
Suppose that condition \eqref{eq:alpha_condition} holds i.e., for all $a\in\ZZ/n\ZZ$
\begin{equation*}
  |\LRI{m}{}{a}{G}-m_c|
  \leq
  \alpha |m_c|
  ,\quad
  |\LRI{m}{}{a}{\Gc}-m_c|
  \leq
  \alpha|m_c|.
\end{equation*}
After expanding the terms of the type $(\LRI{m}{}{a}{G})^{-1}$ or $(1+\LRI{m}{}{a}{G})^{-1}$  around $(m_c)^{-1}$ or $(1+m_c)^{-1}$ respectively, we obtain the following system of linear equations with respect to $\Delta_a:=(\LRI{m}{}{a}{G}-m_c)$ and $\Delta'_a:=(\LRI{m}{}{a}{\Gc}-m_c)$

\begin{EA}{c}
  \frac{1}{wm_c} -\frac{\Delta_a}{wm_c^2}+(1+m_c) +\Delta'_{a-1}-\frac{|z|^2}{w(1+m_c)}+\frac{|z|^2}{w(1+m_c)^2}\Delta_{a+1}+\OO{\frac{\Lambda^2}{|w m_c^3|}} 
    =
    \OO{\varphi^{2Q_{\zeta}}\Psi}
    ,\\
  \frac{1}{wm_c} -\frac{\Delta'_a}{wm_c^2}+(1+m_c) +\Delta_{a+1}-\frac{|z|^2}{w(1+m_c)}+\frac{|z|^2}{w(1+m_c)^2}\Delta'_{a-1}+\OO{\frac{\Lambda^2}{|w m_c^3|}} 
    =
    \OO{\varphi^{2Q_{\zeta}}\Psi}
    .
\end{EA}
Recall, that $m_c$ satisfies the self-consistent equation \eqref{eq:mc_equation}. We end up with the following linear system 
\begin{EA}{c}\EAy\label{eq:lin_sys_1}
    -\frac{\Delta_a}{wm_c^2}+\Delta'_{a-1}+\frac{|z|^2}{w(1+m_c)^2}\Delta_{a+1}
    =
    \OO{\varphi^{2Q_{\zeta}}\Psi}+\OO{\frac{\Lambda^2}{|m_c|}} 
    ,\\ \EAy \label{eq:lin_sys_2}
    -\frac{\Delta'_a}{wm_c^2}+\Delta_{a+1}+\frac{|z|^2}{w(1+m_c)^2}\Delta'_{a-1}
    =
    \OO{\varphi^{2Q_{\zeta}}\Psi}+\OO{\frac{\Lambda^2}{|m_c|}} 
    .
\end{EA}

We introduce the following notation:
\begin{equation*}
  \Delta
  :=
  (\Delta_1,\ldots,\Delta_n,\Delta'_1,\ldots,\Delta'_n)^T
  ,
\end{equation*}
and
\begin{equation*}
  \Gamma_1
  :=
  \begin{pmatrix}
      0 & 0 & \cdots& 0 & \gfr_2 & -\gfr_1
      \\
      -\gfr_1 & 0 & 0 & \cdots & 0 & \gfr_2
      \\
      \gfr_2 & -\gfr_1 & 0 & \cdots & 0 & 0
      \\
      0 & \gfr_2 & -\gfr_1 & 0 & \ddots & 0
      \\
        & \ddots &  & \ddots &  & 
      \\
      0 & \ldots & 0 & \gfr_2 & -\gfr_1 & 0  
    \end{pmatrix}
    ,\quad
  \Gamma
  :=
  \begin{pmatrix}
  I_n&\Gamma_1
  \\
  \Gamma_1^T & I_n
  \end{pmatrix}
  ,
\end{equation*}
where
\begin{equation*}
  \gfr_1
  :=
  \frac{1}{wm_c^2}
  ,\quad
  \gfr_2
  :=
  \frac{|z|^2}{w(1+m_c)^2}
  .
\end{equation*}

Thus we can rewrite the system \eqref{eq:lin_sys_1}-\eqref{eq:lin_sys_2} as
\begin{equation}\label{eq:lin_sys_3}
  \Gamma \Delta
  =
  \OO{\varphi^{2Q_{\zeta}}\Psi}+\OO{\frac{\Lambda^2}{|m_c|}} 
  .
\end{equation}

We have the following proposition about the behaviour of the inverse of $\Gamma$, which is proven in the Appendix.
\begin{pr}\label{pr:Gamma_inv}
There exist $C,\tilde{\tau},\varepsilon>0$ such that the following holds
    
\emph{Case 1:} if $|z|\geq1+\tau$, $|w-\lambda_{\pm}|\leq \tau$ and $\eta \geq \frac{\varphi^C}{N|m_c|}$, then
\begin{equation}\label{eq:Gamma_inv_1}
  \|\Gamma^{-1}\|\sim |w-\lambda_{\pm}|^{-1/2}
\end{equation}

\emph{Case 2:} if $|z|\leq 1-\tau$, $|w-\lambda_{+}|\leq \tau$ and $\eta \geq \frac{\varphi^C}{N|m_c|}$, then
\begin{equation*}
  \|\Gamma^{-1}\|\sim |w-\lambda_{+}|^{-1/2}
\end{equation*}

\emph{Case 3:} if $ \tau \leq |z|\leq 1-\tau$, $|w|\leq \tilde{\tau}$ and $\eta \geq \frac{\varphi^C}{N|m_c|}$, then
\begin{equation*}
  \|\Gamma^{-1}\|\leq C
  ;
\end{equation*}


\emph{Case 4:} if $z\in\Zfr_\tau$, $\max\{0,\lambda_-\}+\tilde{\tau} \leq E\leq \lambda_+-\tau$ and $0\leq \eta \leq \varepsilon$, then
\begin{equation*}
  \|\Gamma^{-1}\|\leq C
  ;
\end{equation*}
\end{pr}

\begin{rem}
  From the proof of Proposition~\ref{pr:Gamma_inv} we see that if $z$ and $w$ are close to the origin, then $\|\Gamma^{-1}\|$ behaves like $(\sqrt{|w|}+|z|^2)^{-1}$.
This singularity differs from the singularities that we obtain in the cases 1 and 2 of the above proposition, and the methods of the present article are not sufficient to study the stability of the system \eqref{eq:sce_mg}-\eqref{eq:sce_mgc} in this case.
\end{rem}

Define $\tilde{\Zfr}_{\tau}=\Zfr_\tau \cap \{\tau\leq |z|\}$.
Let $\delta>0$ be such that $\forall z\in\tilde{\Zfr}_{\tau}$ 
\begin{equation}
[(1+\delta)^{-1}\max\{\lambda_-,0\},(1+\delta)\lambda_+]\subset (\max\{\lambda_-,0\}-\tau,\lambda_+ +\tau)
.
\end{equation}

We now study the stability of the system \eqref{eq:sce_mg}-\eqref{eq:sce_mgc}.
We show that for $z\in\tilde{\Zfr}_{\tau}$ and $w\in S_{z,\delta,\tilde{Q}}$ there exists a gap in the range of $\Lambda$ that depends on the error term in \eqref{eq:sce_mg}-\eqref{eq:sce_mgc}.
Similarly to the case of one matrix, we identify three regimes of the range separation.
Note that near the points $\lambda_{\pm}$ we need an estimate for the error term that is decreasing with respect to $\eta$, therefore, later we shall replace the random control parameter $\Psi$ by a deterministic one.

\begin{pr}\label{pr:sep}
  Let $z\in\tilde{\Zfr}_{\tau}$ and $w\in S_{z,\delta,\tilde{Q}}$. Suppose that condition \eqref{eq:alpha_condition} holds
  \begin{equation}
    \Lambda
    \leq
    \alpha |m_c|
    .
  \end{equation}
Suppose that we have a system \eqref{eq:sce_mg}-\eqref{eq:sce_mgc} with the error term bounded by $\tilde{\Psi}$ that satisfies $\tilde{\Psi}|m_c|^{-1}\leq (\log N)^{-1}$.

Then there exists $M>0$ big enough such that the following holds:

\emph{Case 1:} if $|w-\lambda_{\pm}|\geq M^{-1}$ then
\begin{equation}
  \frac{\Lambda}{|m_c|}\leq \sqrt{\frac{\tilde{\Psi}}{|m_c|}}
  \quad\Rightarrow \quad
  \frac{\Lambda}{|m_c|}\leq M\frac{\tilde{\Psi}}{|m_c|}
\end{equation}

\emph{Case 2:} if $|w-\lambda_{\pm}|\leq M^{-1}$ and $|w-\lambda_{\pm}|\geq M^{3/2}\tilde{\Psi}$ then
\begin{equation}
  \Lambda\leq 2M\frac{\tilde{\Psi}}{|w-\lambda_{\pm}|^{1/2}}
  \quad \Rightarrow \quad
  \Lambda\leq M\frac{\tilde{\Psi}}{|w-\lambda_{\pm}|^{1/2}}
\end{equation}

\emph{Case 3:} if $|w-\lambda_{\pm}|\leq M^{-1}$ and $|w-\lambda_{\pm}|\leq M^{3/2}\tilde{\Psi}$ then
\begin{equation}
  \Lambda\leq 2M\sqrt{\tilde{\Psi}}
  \quad \Rightarrow \quad
  \Lambda\leq M\sqrt{\tilde{\Psi}}
\end{equation}
\end{pr}

\begin{proof}
  \begin{bf}
    Case 1.
  \end{bf}
Suppose that $|w-\lambda_{\pm}|\geq M^{-1}$.
Then we are in one of the two last cases of Proposition~\ref{pr:Gamma_inv}.
The condition $\tau\leq |z|$ assures that in these cases the norm of the matrix $\Gamma^{-1}$ is bounded.
If we linearise the system \eqref{eq:sce_mg}-\eqref{eq:sce_mgc} with an $\OO{\tilde{\Psi}}$ error term up to the first order and divide the equations by $|m_c|$, we obtain the following system
\begin{equation}
  \Gamma \left(\frac{\Delta}{|m_c|}\right)
  =
  \OO{\frac{\tilde{\Psi}}{|m_c|}}+\OO{\frac{\Lambda^2}{|m_c|^2}}
  .
\end{equation}
Since $\|\Gamma^{-1}\|_{\infty}$ is bounded, we deduce that for any $1\leq a \leq n$
\begin{equation}
  \frac{\Delta_a}{|m_c|}=\OO{\frac{\tilde{\Psi}}{|m_c|}}+\OO{\frac{\Lambda^2}{|m_c|^2}}
  ,\quad
  \frac{\Delta'_a}{|m_c|}=\OO{\frac{\tilde{\Psi}}{|m_c|}}+\OO{\frac{\Lambda^2}{|m_c|^2}}
  .
\end{equation}
If $\Lambda |m_c|^{-1}\leq (\tilde{\Psi}|m_c|^{-1})^{1/2}$ then for all $1\leq a\leq n$
\begin{equation}
  \frac{\Delta_a}{|m_c|}=\OO{\frac{\tilde{\Psi}}{|m_c|}}
  ,\quad
  \frac{\Delta'_a}{|m_c|}=\OO{\frac{\tilde{\Psi}}{|m_c|}}
  ,  
\end{equation}
which implies that $\Lambda|m_c|^{-1}=\OO{\tilde{\Psi}|m_c|^{-1}}$.

\begin{bf}
  Case 2.
\end{bf}
Suppose now that $|w-\lambda_{\pm}|\leq M^{-1}$. It follows from Proposition~\ref{pr:Gamma_inv} that $\|\Gamma^{-1}\|_{\infty}=\OO{|w-\lambda_{\pm}|^{-1/2}}$. If we linearise the system \eqref{eq:sce_mg}-\eqref{eq:sce_mgc} up to the first order  we get that
\begin{equation}
  \Delta_a
  =
  \OO{\frac{\tilde{\Psi}}{\sqrt{|w-\lambda_{\pm}|}}}+\OO{\frac{\Lambda^2}{\sqrt{|w-\lambda_{\pm}|}|m_c|}}
  ,\quad
  \Delta'_a=\OO{\frac{\tilde{\Psi}}{\sqrt{|w-\lambda_{\pm}|}}}+\OO{\frac{\Lambda^2}{\sqrt{|w-\lambda_{\pm}|}|m_c|}}
  .
\end{equation}
If $\Lambda\leq 2M\tilde{\Psi}|w-\lambda_{\pm}|^{-1/2}$, then we can rewrite the above equations as
\begin{equation}
  \Delta_a=\OO{\frac{\tilde{\Psi}}{\sqrt{|w-\lambda_{\pm}|}} +M^2\frac{\tilde{\Psi}^2}{|w-\lambda_{\pm}|^{3/2}|m_c|}}
  ,\quad
  \Delta'_a=\OO{\frac{\tilde{\Psi}}{\sqrt{|w-\lambda_{\pm}|}}+M^2\frac{\tilde{\Psi}^2}{|w-\lambda_{\pm}|^{3/2}|m_c|}}
  .
\end{equation}
We now need to show that for $M$ large enough $M^2 \tilde{\Psi}|w-\lambda_{\pm}|^{-1}\leq M$.
But this follows from the condition
\begin{equation}
  |w-\lambda_{\pm}|\geq M^{3/2}\tilde{\Psi}
  .
\end{equation}
Thus Case~2 is established.

\begin{bf}
  Case 3.
\end{bf}
From \eqref{eq:app_2} in the Appendix we know that if $w=\lambda_{\pm}$, then $1-\gfr_1+\gfr_2=0$.
If we rewrite the system \eqref{eq:lin_sys_1}-\eqref{eq:lin_sys_2} using \eqref{eq:prop_mc_lambda_1} and \eqref{eq:prop_mc_lambda_2} we get for $a\in \ZZ/n\ZZ$
\begin{EA}{l}
 \Delta'_{a-1}-\frac{1}{\lambda_{\pm}m_c^2(\lambda_{\pm})}\Delta_a+\frac{|z|^2}{\lambda_{\pm}(1+m_c(\lambda_{\pm}))^2}\Delta_{a+1}
   =
   \OO{\tilde{\Psi}+\Lambda^2+\Lambda|w-\lambda_{\pm}|^{1/2}}
   ,\\
 \Delta_{a+1}-\frac{1}{\lambda_{\pm}m_c^2(\lambda_{\pm})}\Delta'_a+\frac{|z|^2}{\lambda_{\pm}(1+m_c(\lambda_{\pm}))^2}\Delta'_{a-1}
   =
   \OO{\tilde{\Psi}+\Lambda^2+\Lambda|w-\lambda_{\pm}|^{1/2}}   
   .
\end{EA}
Consider now the matrix $\Gamma(\lambda_{\pm})$.
We will show that $\mathrm{rank} \Gamma(\lambda_{\pm})=2n-1$.
From \eqref{eq:app_2} it follows that $l_n(I-\Gamma_1\Gamma_1^{T})$ (the $n$th eigenvalue of $(I-\Gamma_1\Gamma_1^{T})(\lambda_{\pm})$ defined in \eqref{eq:app_1}) is equal to zero.
From the formula \eqref{eq:app_1} we have that for $1\leq j\leq n-1$
\begin{equation}
  l_n(I-\Gamma_1\Gamma_1^T)-l_j(I-\Gamma_1\Gamma_1^T)
  =
  \gfr_1\gfr_2 2(1-\Re e^{j \sqrt{-1}2\pi/n})
  .
\end{equation}
Since $\gfr_1\sim \gfr_2 \sim 1$ we deduce that $(I-\Gamma_1\Gamma_1^T)(\lambda_{\pm})$ has only one vanishing eigenvalue.
Using the formula for the determinant of block matrices we have that
\begin{equation}
  \det (\Gamma-xI)
  =
  \det (I-\Gamma_1\Gamma_1^T -(2x-x^2)I)
  .
\end{equation}
Therefore, if $l_j$ is an eigenvalue of $I-\Gamma_1\Gamma_1^T$, then $1-\sqrt{1-l_j}$ and $1+\sqrt{1-l_j}$ are eigenvalues of $\Gamma$. 
We conclude that $\Gamma(\lambda_{\pm})$ has $2n-1$ non-zero eigenvalues and that $\mathrm{rank}\Gamma(\lambda_{\pm})= 2n-1$.
The sum of each row or column of $\Gamma(\lambda_{\pm})$ is equal to zero, which implies that $\mathrm{Ker} \Gamma =\{t (1,1,\ldots,1)^T,\, t\in\RR\}$. 
Suppose for simplicity that the lower right $n-1$-minor of $\Gamma(\lambda_{\pm})$ is invertible and denote this minor by $\tilde{\Gamma}$.
Then
\begin{equation}
  \Gamma \Delta -\Gamma  (\Delta_1,\Delta_1,\ldots,\Delta_1)^T
  =
  \OO{\tilde{\Psi}+\Lambda^2+\Lambda|w-\lambda_{\pm}|^{1/2}}   
\end{equation}
gives a system of $n$ linear equations with respect to $\tilde{\Delta}:=(\Delta_2-\Delta_1,\ldots,\Delta_n-\Delta_1,\Delta'_1-\Delta_1,\ldots,\Delta'_n-\Delta_1)$.
The system
\begin{equation}
  \tilde{\Gamma}\tilde{\Delta}
  =
  \OO{\tilde{\Psi}+\Lambda^2+\Lambda|w-\lambda_{\pm}|^{1/2}}   
\end{equation}
can be solved, which implies that
\begin{equation}
  \max_{1\leq a\leq n} |\Delta_a-\Delta_1|=\OO{\tilde{\Psi}+\Lambda^2+\Lambda|w-\lambda_{\pm}|^{1/2}}   
  ,\quad
  \max_{1\leq a\leq n} |\Delta'_a-\Delta_1|=\OO{\tilde{\Psi}+\Lambda^2+\Lambda|w-\lambda_{\pm}|^{1/2}}   
\end{equation}
and also
\begin{equation}\label{eq:sep_1}
  \max_{1\leq a\leq n} |\Delta_a^2-\Delta_1^2|=\OO{\tilde{\Psi}\Lambda+\Lambda^3+\Lambda^2|w-\lambda_{\pm}|^{1/2}}   
  ,\quad
  \max_{1\leq a\leq n} |(\Delta'_a)^2-\Delta_1^2|=\OO{\tilde{\Psi}\Lambda+\Lambda^3+\Lambda^2|w-\lambda_{\pm}|^{1/2}}   
\end{equation}

We now linearise the system \eqref{eq:sce_mg}-\eqref{eq:sce_mgc} with an error term bounded by $\tilde{\Psi}$ up to the second order and expand the function $m_c$ around $\lambda_{\pm}$ according to \eqref{eq:prop_mc_lambda_1} and \eqref{eq:prop_mc_lambda_2}.
We end up with the following system for $a\in\ZZ/n\ZZ$
\begin{EA}{l}
     \Delta'_{a-1}-\frac{1}{\lambda_{\pm}m_c^2}\Delta_a+\frac{|z|^2}{\lambda_{\pm}(1+m_c)^2}\Delta_{a+1}+\frac{1}{\lambda_{\pm}m_c^3}\Delta_a^2 -\frac{|z|^2}{\lambda_{\pm}(1+m_c)^3}\Delta_{a+1}^2
   =
   \OO{\tilde{\Psi}+\Lambda^3+\Lambda\sqrt{|w-\lambda_{\pm}|}}
   \\
   \Delta_{a+1}-\frac{1}{\lambda_{\pm}m_c^2}\Delta'_a+\frac{|z|^2}{\lambda_{\pm}(1+m_c)^2}\Delta'_{a-1}+\frac{1}{\lambda_{\pm}m_c^3}(\Delta'_a)^2 -\frac{|z|^2}{\lambda_{\pm}(1+m_c)^3}(\Delta'_{a-1})^2
   =
   \OO{\tilde{\Psi}+\Lambda^3+\Lambda\sqrt{|w-\lambda_{\pm}|}}
   .
\end{EA}
Adding all these equations we get an equation for the sum of squares of $\Delta_a$ and $\Delta'_a$
\begin{EA}{rl}
   (1-\frac{1}{\lambda_{\pm}m_c^2}+\frac{|z|^2}{\lambda_{\pm}(1+m_c)^2})(\sum_a\Delta'_{a}+\Delta_a)+(\frac{1}{\lambda_{\pm}m_c^3}-\frac{|z|^2}{\lambda_{\pm}(1+m_c)^3})&(\sum_a \Delta_a^2+(\Delta'_{a})^2)
   \\
   =
   (\frac{1}{\lambda_{\pm}m_c^3}-\frac{|z|^2}{\lambda_{\pm}(1+m_c)^3})(\sum_a \Delta_a^2+(\Delta'_{a})^2)
   &=
   \OO{\tilde{\Psi}+\Lambda^3+\Lambda\sqrt{|w-\lambda_{\pm}|}}
\end{EA}
where we used \eqref{eq:app_2}.
Suppose that
\begin{equation}\label{eq:sep_2}
\frac{1}{\lambda_{\pm}m_c^3}-\frac{|z|^2}{\lambda_{\pm}(1+m_c)^3}
\end{equation}
 does not vanish.
Then, using \eqref{eq:sep_1} we have that
\begin{equation}
  \Delta_1^2
  =
  \OO{\tilde{\Psi}+\Lambda^3+\Lambda\sqrt{|w-\lambda_{\pm}|} +\tilde{\Psi}\Lambda+\Lambda^3+\Lambda^2|w-\lambda_{\pm}|^{1/2}} 
  .
\end{equation}
If $\Lambda\leq 2M\sqrt{\tilde{\Psi}}$ and $|w-\lambda_{\pm}|\leq M^{3/2}\tilde{\Psi}$, then
\begin{equation}
  \Delta_1^2
  =
  \OO{\tilde{\Psi}+ 2 M^{7/4}\tilde{\Psi} + \tilde{\Psi}^{3/2}}
\end{equation}
and from \eqref{eq:sep_1} we have
\begin{equation}
  \max_a\{|\Delta_a|^2,|\Delta'_a|^2\}
  =
  \OO{\tilde{\Psi}+ 2 M^{7/4}\tilde{\Psi} + \tilde{\Psi}^{3/2}}
  .
\end{equation}
We conclude that $\Lambda \leq C M^{7/8} \sqrt{\tilde{\Psi}}\leq M\sqrt{\tilde{\Psi}}$ for $M$ big enough.
The last thing to show is that \eqref{eq:sep_2} is non-zero for any $z\in \tilde{\Zfr}_{\tau}$.
This follows from \eqref{eq:app_3} and \eqref{eq:app_4} in the Appendix.
The proposition is thus proven.
\end{proof}

In the following proposition we estimate $\Lambda$ in the case when $\eta$ is of order $\OO{1}$.
The beginning of the proof is similar to the proof of \cite[Lemma~17]{NemiNHPO}, but for the reader's convenience we provide this proof here with all the details.
\begin{pr}\label{pr:large}
Let $\eta_0>0$. 
Then for any $z\in\tilde{\Zfr}_{\tau}$ and for any $w\in S_{z,\delta,\tilde{Q}_{\zeta}}\cap\{\eta=\eta_0\}$
\begin{equation}
  \sup_{w\in S_{z,\delta,\tilde{Q}_{\zeta}}\cap\{\eta=\eta_0\}}\Lambda
  \leq
  \frac{\varphi^{Q_{\zeta}}}{\sqrt{N}}
\end{equation}
with $\zeta$-high probability.
\end{pr}
\begin{proof}
First of all recall that by \eqref{eq:bord_guion_minor_estimate}
  \begin{equation*}
    |\LRI{m}{}{a}{G}-\LRI{m}{(i,\emptyset)}{a}{G}|
    \leq
    \frac{4}{N\eta}
    ,\quad
    |\LRI{m}{}{a}{\Gc}-\LRI{m}{(i,\emptyset)}{a}{\Gc}|
    \leq
    \frac{4}{N\eta}.
  \end{equation*}
Therefore, $\LRI{m}{}{a}{G}$ and $\LRI{m}{}{a}{\Gc}$ as functions of the columns $\xb_{k},1\leq k\leq nN$ satisfy the condition \eqref{eq:mcdiarmid_condition} for any $w\in S_{z,\delta,\tilde{Q}_{\zeta}}\cap\{\eta=\eta_0\}$ and we can apply the McDiarmid's concentration inequality, so that
  \begin{equation*}
    \Pr{|\LRI{m}{}{a}{G}-\Me{\LRI{m}{}{a}{G}}|\geq t}
    \leq
    C e^{-c t^2 N}
  \end{equation*}
and similarly for $\LRI{m}{}{a}{\Gc}$.
If we take $t=c^{-1/2}\varphi^{\zeta/2}N^{-1/2}$ in the above inequality we get that for any $w\in S_{z,\delta,\tilde{Q}_{\zeta}}\cap\{\eta=\eta_0\}$
  \begin{equation*}
    |\LRI{m}{}{a}{G}-\Me{\LRI{m}{}{a}{G}}|
    =
    \OO{\frac{\varphi^{\zeta/2}}{\sqrt{N}}}
    ,\quad
    |\LRI{m}{}{a}{\Gc}-\Me{\LRI{m}{}{a}{\Gc}}|
    =
    \OO{\frac{\varphi^{\zeta/2}}{\sqrt{N}}}
  \end{equation*}
with $\zeta$-high probability.

Let $\hat{X}$ be a $nN\times nN$ random matrix having the same block structure as the matrix $X$ but with \emph{iid} non-zero entries.
Suppose that $\LRI{X}{}{a,a+1}{kl}$ has the same distribution as $\LRI{X}{}{12}{11}$ for all $a\in \ZZ/n\ZZ$ and $1\leq k,l\leq N$.
Denote by $\hat{G}$ and $\hat{\Gc}$ corresponding resolvent matrices.
In was shown in \cite[Lemma~3]{NemiNHPO} that 
\begin{equation*}
  |\Me{\LRI{m}{}{a}{\hat{G}}}-\Me{\LRI{m}{}{a}{G}}|
  =
  \OO{\frac{\varphi^{\zeta}}{\sqrt{N}}}
  ,\quad
  |\Me{\LRI{m}{}{a}{\hat{\Gc}}}-\Me{\LRI{m}{}{a}{\Gc}}|
  =
  \OO{\frac{\varphi^{\zeta}}{\sqrt{N}}}
  ,
\end{equation*}
and from \cite[Lemma~14]{OrouSosh} we know that
  \begin{equation*}
  \Me{\LRI{m}{}{a}{\hat{G}}}
  =
  \Me{m_{\hat{G}}}
  ,\quad
  \Me{\LRI{m}{}{a}{\hat{\Gc}}}
  =
  \Me{m_{\hat{\Gc}}}
  .
  \end{equation*}
Therefore,
\begin{equation}\label{eq:large_eta_1}
  |\LRI{m}{}{a}{G}-m_G|
  =
  \OO{\frac{\varphi^{\zeta/2}}{\sqrt{N}}}
  =
  \OO{\frac{\varphi^{Q_{\zeta}}}{\sqrt{N}}}
  ,\quad
  |\LRI{m}{}{a}{\Gc}-m_G|
  =
  \OO{\frac{\varphi^{\zeta/2}}{\sqrt{N}}}
  =
  \OO{\frac{\varphi^{Q_{\zeta}}}{\sqrt{N}}}
  ,
\end{equation}
where we used that $m_G=m_{\Gc}$ and that by definition of $Q_{\zeta}$ (see the proof of Lemma~\ref{lem:lde}) $Q_{\zeta}>\zeta$.
With the same argument as in Theorem~\ref{thm:sce} we can thus show that with $\zeta$-high probability
  \begin{equation}\label{eq:large_1}
    \frac{1}{G_{ii}}
    =
    -w(1+m_G)+\frac{|z|^2}{1+m_G+\OO{\varphi^{Q_{\zeta}}N^{-1/2}}}+\OO{\varphi^{2Q_{\zeta}}N^{-1/2}}
    .
  \end{equation}
Indeed,
\begin{equation*}
  \frac{1}{G_{ii}}
  =
   -w(1+\LRI{m}{(i(a),\emptyset)}{a-1}{\Gc})+\OO{\varphi^{2Q_{\zeta}}N^{-1/2}}+\frac{|z|^2}{1+\LRI{m}{(i(a),i(a))}{a+1}{G}+\OO{\varphi^{Q_{\zeta}}N^{-1/2}}}
   .
\end{equation*}
But from the relations \eqref{eq:bord_guion_minor_estimate} and \eqref{eq:large_eta_1} we have
\begin{EA}{ll}
  \LRI{m}{(i(a),\emptyset)}{a-1}{\Gc}
  &= m_G +(m_{\Gc}-m_G)+(\LRI{m}{}{a-1}{\Gc}-m_{\Gc})+(\LRI{m}{(i(a),\emptyset)}{a-1}{\Gc}-\LRI{m}{}{a-1}{\Gc})
  \\
  &=m_G + 0 + \OO{\frac{\varphi^{Q_{\zeta}}}{\sqrt{N}}} + \OO{\frac{1}{N}} 
  =
  m_G+\OO{\frac{\varphi^{Q_{\zeta}}}{\sqrt{N}}}
  ,
\end{EA}
and similarly $\LRI{m}{(i(a),i(a))}{a+1}{G}=m_G+\OO{\varphi^{Q_{\zeta}}N^{-1/2}}$.

Repeating the proof of \cite[Lemma~6.12]{BourYauYin} we can show that $|1+m_G|$ is large enough with respect to the error term of order $\OO{\varphi N^{-1/2}}$.
Therefore we rewrite \eqref{eq:large_1} as 
\begin{equation}
  \frac{1}{G_{ii}}
  =
   -w(1+m_G)+\frac{|z|^2}{1+m_G}+\OO{\varphi^{2Q_{\zeta}}N^{-1/2}}
\end{equation}
The entries of the resolvent matrix are bounded by $\eta^{-1}$.
We end up with the following equation
\begin{equation}
  G_{ii}
  =
  \left[ -w(1+m_G)+\frac{|z|^2}{1+m_G}\right]^{-1}+\OO{\varphi^{2Q_{\zeta}}N^{-1/2}}
\end{equation}

Now we can conclude as in \cite[Lemma~6.12]{BourYauYin} that 
\begin{equation*}
  \sup_{w\in S_{z,\delta,\tilde{Q}_{\zeta}}\cap\{\eta=\eta_0\}}|m_G(z,w)-m_c(z,w)|
  =
  \OO{\varphi^{2Q_{\zeta}}N^{-1/2}}
\end{equation*}
with $\zeta$-high probability. The result follows using \eqref{eq:large_eta_1}.
\end{proof}

Now, following \cite{BourYauYin}, we establish a preliminary estimate for $\Lambda$ that shows that the bound \eqref{eq:alpha_condition} holds for any $z\in\tilde{\Zfr}_{\tau}$ and $w\in S_{z,\delta,S_{\tilde{Q}_{\zeta}}}$. 
We shall use Proposition~\ref{pr:sep} with the function
\begin{equation}
  \varphi^{2Q_{\zeta}}\sqrt{\frac{|w|^{-1/2}}{N\eta}}
\end{equation}
as $\tilde{\Psi}$.
\begin{thm}\label{thm:continuity_argument}
  For any $\zeta>0$ and any $z\in\tilde{\Zfr}_{\tau}$ 
  \begin{equation}
    \sup_{w\in S_{z,\delta,\tilde{Q}_{\zeta}}}\Lambda
    =
    \OO{\varphi^{2Q_{\zeta}}|w|^{-1/2}\left(\frac{|w|^{1/2}}{N\eta}\right)^{1/4}}
  \end{equation}
\end{thm}
\begin{proof}
  Following the approach used by Bourgade, Yau and Yin in \cite{BourYauYin}, we prove the theorem in two steps.
Firstly we show that  with $\zeta$-high probability the bound 
\begin{equation}\label{eq:weak_concentration1}
  \Lambda(z,w)
    =
    \OO{\varphi^{2Q_{\zeta}} |w|^{-1/2}\left(\frac{|w|^{1/2}}{N\eta}\right)^{1/4}}
\end{equation}
holds for $N^{-K}$-net in $\tilde{S}_{\tilde{Q}_{\zeta}}$ for $K>0$ big enough.
Next, we use the continuity properties of $\Lambda$ to extend the result to the whole set  $\tilde{S}_{\tilde{Q}_{\zeta}}$.

The first part is a bootstrapping type argument.
We fix $E$ and consider firstly $\eta=\OO{1}$, for which \eqref{eq:weak_concentration1} holds by Proposition~\ref{pr:large}.
Then we show that if we decrease $\eta$ then condition \eqref{eq:alpha_condition} still holds, and thus we can apply Proposition~\ref{pr:sep} to get a gap in the range of possible values of $\Lambda$.
From the continuity properties of $\Lambda$ we deduce that $\Lambda$ stays below the gap and that the weak estimate \eqref{eq:weak_concentration1} holds for this smaller choice of $\eta$.
We continue these iterations as long as $E+\sqrt{-1}\eta$ stays in $S_{z,\delta,\tilde{Q}_{\zeta}}$.
We now provide the detailed proof.

Let $K>0$ and let $\eta>0$ such that $E+\sqrt{-1}(\eta-N^{-K})\in S_{z,\delta,\tilde{Q}_{\zeta}}$.
As we fix $z$ and $E$, we can introduce a simplified notation $\Lambda(\eta):=\Lambda(z,E+\sqrt{-1}\eta)$.
Suppose firstly that we are not in the neighbourhood of $\lambda_{\pm}$ and that
\begin{equation}
  \Lambda(\eta)
  =
  \OO{\varphi^{2Q_{\zeta}}\left(\frac{|w|^{-1/2}}{N\eta}\right)^{1/2}}
  .
\end{equation}
Then 
\begin{EA}{rcl}
  \Lambda(\eta-N^{-K})
  &\,\leq&\,
  \max_{a\in \ZZ/n\ZZ} \left(|\LRI{m}{}{a}{G}(\eta)-m_c(\eta)|+|\LRI{m}{}{a}{\Gc}(\eta)-m_c(\eta)|\right)
  \\
  &&+\max_{a\in \ZZ/n\ZZ} \left(|\LRI{m}{}{a}{G}(\eta)-\LRI{m}{}{a}{G}(\eta-N^{-K})|+|\LRI{m}{}{a}{\Gc}(\eta)-\LRI{m}{}{a}{\Gc}(\eta-N^{-K})|\right)
  + |m_c(\eta)-m_c(\eta-N^{-K})|
  \\
  &\,\leq&\,
  \OO{\varphi^{2Q_{\zeta}}\left(\frac{|w|^{-1/2}}{N\eta}\right)^{1/2}} 
  \\\EAy \label{eq:lambda_bootstrap}
  &&+ N^{-K}\sup_{w\in S_{z,\delta,\tilde{Q}_{\zeta}}}\max_{a\in \ZZ/n\ZZ}\left(\left|\frac{\partial \LRI{m}{}{a}{G}}{\partial \eta}\right|+\left|\frac{\partial \LRI{m}{}{a}{\Gc}}{\partial \eta}\right|\right)+N^{-K}\sup_{w\in S_{z,\delta,\tilde{Q}_{\zeta}}}\left|\frac{\partial m_c}{\partial \eta}\right|
  .
\end{EA}
Note that
\begin{equation}
  \left(\frac{|w|^{-1/2}}{N\eta}\right)^{1/2}
  =
  |w|^{-1/2}\left(\frac{1}{|w|^{-1/2}N\eta}\right)^{1/2}
  \leq
   |m_c|\left(\frac{1}{|m_c|N\eta}\right)^{1/4}
\end{equation}
From the definition of $S_{z,\delta,\tilde{Q}_{\zeta}}$ (see \eqref{eq:s_set}) we have that on this set $\eta\geq N^{-2}$. According to Lemma~\ref{lem:derivative_estimates} if we take $K>0$ big enough then there exists $N_0\in\NN$ such that for all $N\geq N_0$
\begin{equation}\label{eq:weak_concentration2}
  \Lambda(\eta-N^{-K})
  \leq
  \alpha|m_c(\eta-N^{-K})|
  .
\end{equation}
Moreover, from the boundedness of $|m_c|^{-1}\sim |w|^{1/2}$ and the fact that
\begin{equation*}
  \left|\frac{\partial}{\partial \eta} (|w|^{1/2}\eta)^{-1/2}\right|
  =
  \OO{\eta^{-7/4}}
\end{equation*}
we obtain that 
\begin{equation}\label{eq:weak_concentration3}
 \frac{\Lambda(\eta-N^{-K})}{|m_c(\eta-N^{-K})|}
  \leq
  C\left(\varphi^{2Q_{\zeta}} \left(\frac{|E+\sqrt{-1}(\eta-N^{-K})|^{-1/2}}{N(\eta-N^{-K})}\right)^{1/2}\frac{1}{m_c(\eta-N^{-K})}\right)^{1/2}
  .
\end{equation}
By \eqref{eq:weak_concentration2} we see that Proposition~\ref{pr:sep} can be applied to $\Lambda(\eta-N^{-K})$, and by \eqref{eq:weak_concentration3} we see that
\begin{equation*}
  \Lambda(\eta-N^{-K})
  =
  \OO{\varphi^{2Q_{\zeta}} \left(\frac{|E+\sqrt{-1}(\eta-N^{-K})|^{-1/2}}{N(\eta-N^{-K})}\right)^{1/2}}
  .
\end{equation*}
Suppose now that we are close to $\lambda_{\pm}$ and that we have
\begin{equation}
  \Lambda(\eta)
  =
  \OO{\varphi^{2Q_{\zeta}}\left(\frac{|w|^{-1/2}}{N\eta}\right)^{1/4}}
  .
\end{equation}
Note that in this case $|w|\sim 1$.
Therefore, as in \eqref{eq:lambda_bootstrap} we have that
\begin{EA}{rcl}
  \Lambda(\eta-N^{-K})
  &\,\leq&\,
  \OO{\varphi^{2Q_{\zeta}}\left(\frac{1}{N\eta}\right)^{1/4}} + N^{-K}\sup_{w\in S_{z,\delta,\tilde{Q}_{\zeta}}}\left(\max_{a\in \ZZ/n\ZZ}\left(\left|\frac{\partial \LRI{m}{}{a}{G}}{\partial \eta}\right|+\left|\frac{\partial \LRI{m}{}{a}{\Gc}}{\partial \eta}\right|\right)+ \left|\frac{\partial m_c}{\partial \eta}\right|  \right)
  \\
  &\,\leq&\,
  \alpha|m_c(\eta-N^{-K})|
\end{EA}
for $N$ sufficiently large.
Again, by Proposition~\ref{pr:sep} and continuity of $\Lambda$ we get that
\begin{equation}
  \Lambda(\eta-N^{-K})
  =
  \OO{\varphi^{2Q_{\zeta}}\left(\frac{|E+\sqrt{-1}(\eta-N^{-K})|^{-1/2}}{N(\eta-N^{-K})}\right)^{1/4}}
  .
\end{equation}
We showed that there exists $K>0$ such that if \eqref{eq:weak_concentration1} holds for $E+\sqrt{-1}\eta\in S_{z,\delta,\tilde{Q}_{\zeta}}$ with $\zeta$-high probability, then with $\zeta$-high probability \eqref{eq:weak_concentration1} holds for $E+\sqrt{-1}(\eta-N^{-K})$ with the same constant in $\OO{\, }$, as long as $E+\sqrt{-1}(\eta-N^{-K})\in S_{z,\delta,\tilde{Q}_{\zeta}}$.
Let $K>0$ and let $\Theta_N(K):=\{k N^{-K}+\sqrt{-1} l N^{-K} \,|\,k,l\in\ZZ\}\subset\CC$. 
From Proposition~\ref{pr:large} we know that \eqref{eq:weak_concentration1} holds for any $w\in S_{z,\delta,\tilde{Q}_{\zeta}}\cap\{\eta = \varepsilon\}$.
Starting from $w\in\ S_{z,\delta,\tilde{Q}_{\zeta}}\cap \Theta(K)$ which are close to $\{\eta=\varepsilon\}$, we can step by step decrease the imaginary part of $w$ and  show that the bound \eqref{eq:weak_concentration1} holds for all $w\in S_{z,\delta,\tilde{Q}_{\zeta}}$ with $\zeta$-high probability.
We can finish the proof by using Lemma~\ref{lem:derivative_estimates} and continuity properties of $m_c$ to extend \eqref{eq:weak_concentration1} to the whole set $S_{z,\delta,\tilde{Q}_{\zeta}}$.

\end{proof}

The important consequence of Theorem~\ref{thm:continuity_argument} is that for $z\in \Zfr_\tau$ and $w\in S_{z,\delta,\tilde{Q}_{\zeta}}$ with $\zeta$-high probability all the approximate equations \eqref{eq:sce_1}-\eqref{eq:sce_mgc} hold.
We can expand this set of relations by adding the approximate equations for the individual entries of the resolvent matrices for the minors.
\begin{cor}
\label{cor:weak_concentration_entries}
Let $\TT,\UU\subset \{1,\ldots,nN\}$ such that $|\TT|+|\UU|\leq \varphi^{2Q_{\zeta}}$.
For any $\zeta,\tau>0$ for any $z\in\Zfr_{\tau}$ and $w\in\tilde{S}_{z,\delta,\tilde{Q}_{\zeta}}$ with $\zeta$-high probability the following holds

  \begin{EA}{ll} \EAy \label{eq:weak_concentration_entries_1}
    \LRI{G}{(\TT,\UU)}{a}{ii}
    =
    \frac{1}{-w(1+\LRI{m}{}{a-1}{\Gc})}+\OO{\varphi^{2Q_{\zeta}}\Psi}
    ,
    &\quad
    \mbox{ if }
    \quad
    i(a)\in \UU
    ,
    \\ \EAy \label{eq:weak_concentration_entries_2}
    \LRI{G}{(\TT,\UU)}{a}{ii}
    =
    \left[-w(1+\LRI{m}{}{a-1}{\Gc})+\frac{|z|^2}{1+\LRI{m}{}{a+1}{G}} \right]^{-1}+\OO{\varphi^{2Q_{\zeta}}\Psi}
    ,
    &\quad
    \mbox{ if }
    \quad
    i(a)\notin \UU
    ,
    \\ \EAy \label{eq:weak_concentration_entries_3}
    G_{kl}^{(\TT,\UU)}
    =
    \OO{\varphi^{2Q_{\zeta}}\Psi}
    ,
    &\quad
    \mbox{ if }
    \quad
    k\neq l
    ,
  \end{EA}
and similarly when changing the rôles of $G,\UU$ and $\Gc, \TT$.
\end{cor}
\begin{proof}
  We use a similar argument as in the proof of Corollary~1 in \cite{NemiNHPO}, that relies on the Schur complement formula, Theorem~\ref{thm:continuity_argument} and lemmas \ref{lem:sce_schur}, \ref{lem:prop_of_mc_close_functions}, \ref{lem:lde} and \ref{lem:bord_guion_minor_estimate}.
\end{proof}


\subsection{Strong concentration}

In this section we finish the proof of Theorem~\ref{thm:conc_main}.
Note that due to Theorem~\ref{thm:continuity_argument} the initial bound for $\Lambda$ \eqref{eq:alpha_condition} holds on the set $S_{z,\delta,\tilde{Q}_{\zeta}}$ with $\zeta$-high probability, and thus on this set Theorem~\ref{thm:sce} and Proposition~\ref{pr:sep} hold.
Recall, that in the proof of Theorem~\ref{thm:continuity_argument} we used the stability of the system of approximate equations
\eqref{eq:sce_mg}-\eqref{eq:sce_mgc} to obtain the following estimates for $\Lambda$ depending on the error term in the self-consistent equations:
\newline
\emph{suppose that the error terms in \eqref{eq:sce_mg}-\eqref{eq:sce_mgc} are bounded by $\tilde{\Psi}$, which is a deterministic function strictly decreasing in $\eta$ near the points $w=\lambda_{\pm}$; suppose that $\Lambda\leq \tilde{\Psi}$ for some $\eta=\OO{1}$; then with $\zeta$-high probability for any $w\in S_{z,\delta,\tilde{Q}_{\zeta}}$
\begin{itemize}
\item[(i)] if $w$ is far enough from $\lambda_{\pm}$, then $\Lambda\leq \tilde{\Psi}$,
\item[(ii)] if $w$ is in the neighbourhood of $\lambda_{\pm}$, then $\Lambda \leq (\tilde{\Psi})^{1/2}$.
\end{itemize}}
Therefore, improving the bound of the error terms in \eqref{eq:sce_mg}-\eqref{eq:sce_mgc} will lead us to a better estimate of $\Lambda$.

Following the idea from \cite{BourYauYin}, we can use Theorem~\ref{thm:continuity_argument} to obtain the system of second order self-consistent equation 
\begin{EA}{ll}
  -\frac{1}{w\LRI{m}{}{a}{G}}
  =&
  1+\LRI{m}{}{a-1}{\Gc}-\frac{|z|^2}{w(1+\LRI{m}{}{a+1}{G})} +\Ec_a
  ,\\
  -\frac{1}{w\LRI{m}{}{a}{\Gc}}
  =&
  1+\LRI{m}{}{a+1}{G}-\frac{|z|^2}{w(1+\LRI{m}{}{a-1}{\Gc})} +\tilde{\Ec}_a
  ,
\end{EA}
where $\Ec_a:=\OO{\varphi^{4Q_{\zeta}}\Psi^2|m_c|^{-1}+\Ec_a^{(1)}+\Ec_a^{(2)}}$ and 
\begin{EA}{l}
  \Ec_a^{(1)}
  :=
  \frac{1}{N}\sum_{i=1}^{N}\left[\left(\LRI{m}{}{a-1}{\Gc}-\LRI{m}{(i(a),\emptyset)}{a-1}{\Gc}\right)+\left(\LRI{m}{}{a+1}{G}-\LRI{m}{(i(a),i(a))}{a+1}{G}\right) \right]
  ,\\
  \Ec_a^{(2)}
  :=
  \frac{1}{N}\sum_{i=1}^{N}\left[\LRI{\Zc}{}{a}{i}+\LRI{Z}{i(a)}{a}{i}\right]
\end{EA}
and similarly for $\tilde{\Ec}_a$.
In the next two lemmas we show that $\Ec_a^{(1)}$ and $\Ec_a^{(2)}$ are of order $\OO{\varphi^{4Q_{\zeta}}(\Psi[\tilde{\Lambda}])^2|m_c|^{-1}}$, where
\begin{equation}
  \Psi[\tilde{\Lambda}]
  :=
  \sqrt{\frac{\Im m_c +\tilde{\Lambda}}{N\eta}}+\frac{1}{N\eta}
  ,
\end{equation}
and $\tilde{\Lambda}$ is some deterministic estimate for $\Lambda$ satisfying $\tilde{\Lambda}\leq \alpha|m_c|$.
The arguments in the lemmas \ref{lem:strong_error_1} and \ref{lem:strong_error_2} are similar to the proof of Lemma~18 in \cite{NemiNHPO} and  Lemma~7.3 in \cite{BourYauYin} respectively, therefore we provide here only a sketch proof, indicating the ideas that were used, but omitting the technical details.

\begin{lem}\label{lem:strong_error_1}
With $\zeta$-high probability for any $w\in S_{z,\delta,Q_{\zeta}}$
  \begin{equation}
    \sum_{a=1}^n \Ec^{(1)}_a+\tilde{\Ec}^{(1)}_a
    =
    \OO{\varphi^{4Q_{\zeta}}\Psi^2|m_c|^{-1}}
    .
  \end{equation}
\end{lem}
\begin{proof}
  Firstly, we use Lemma~\ref{lem:minor_differences} to rewrite $\Ec^{(1)}_a$ or $\tilde{\Ec}^{(1)}_a$ in a form that is easy to bound using the estimates for the entries of the resolvent matrix obtained in Corollary~\ref{cor:weak_concentration_entries}.
For example, 
\begin{equation}\label{eq:strong_error_1}
  \LRI{m}{}{a+1}{G}-\LRI{m}{(i(a),i(a))}{a+1}{G}
  =
  (\LRI{m}{}{a+1}{G}-\LRI{m}{(i(a),\emptyset)}{a+1}{G})+(\LRI{m}{(i(a),\emptyset)}{a+1}{G}-\LRI{m}{(i(a),i(a))}{a+1}{G})
  ,
\end{equation}
and
\begin{equation}
  \LRI{m}{}{a+1}{G}-\LRI{m}{(i(a),\emptyset)}{a+1}{G}
  =
  \frac{1}{N}\sum_{k=1}^{N}\LRI{G}{(i(a),\emptyset)}{a+1}{kk}-\LRI{G}{}{a+1}{kk}
  =
  \frac{1}{N}\sum_{k=1}^{N}\frac{\LRI{G}{}{a,a+1}{ki}\LRI{G}{}{a+1,a}{ik}}{\LRI{G}{}{a+1}{kk}}
  .
\end{equation}
By Corollary~\ref{cor:weak_concentration_entries} all the off-diagonal entries of $G$ are bounded by $\varphi^{2Q_{\zeta}}\Psi$, and for any $j\in\llbracket 1,nN\rrbracket$ $G_{jj} \sim m_c$.
Therefore we deduce that
\begin{equation}
  \LRI{m}{}{a+1}{G}-\LRI{m}{(i(a),\emptyset)}{a+1}{G}
  =
  \OO{\varphi^{4Q_{\zeta}}\frac{\Psi^2}{|m_c|}}
\end{equation}
with $\zeta$-high probability.

To bound the second term in \eqref{eq:strong_error_1}, we rewrite it using Lemma~\ref{lem:minor_differences} as
\begin{equation}
  \LRI{m}{(i(a),\emptyset)}{a+1}{G}-\LRI{m}{(i(a),i(a))}{a+1}{G}
  =
  \frac{1}{N}\sum_{k=1}^{N}\frac{\left[(G^{(i(a),i(a))}y_{i(a)}^*)(y_{i(a)}G^{(i(a),i(a)})\right]_{k(a+1),k(a+1)}}{1+y_{i(a)}G^{(i(a),i(a))} y_{i(a)}^*}
  .
\end{equation}
Proceeding as in the proof of Lemma~18 in \cite{NemiNHPO} and using the estimates of Corollary~\ref{cor:weak_concentration_entries}, we can show that this term is bounded by $\varphi^{4Q_{\zeta}}\Psi^2|m_c|^{-1}$.

All the other estimates follow using a similar argument.
\end{proof}

\begin{lem}\label{lem:strong_error_2}
Let $\zeta>0$.
Suppose that for any $w\in S_{z,\delta,\tilde{Q}_{\zeta}}$
\begin{equation}
  \Lambda
  \leq
  \tilde{\Lambda}
  \leq
  \alpha |m_c|
\end{equation}
with probability at least $1-e^{-p_N}$, where $\varphi\leq p_N \leq \varphi^{2\zeta}$.
Then with probability at least $1-e^{-p_N(\log N)^{-2}}$ for any $w\in S_{z,\delta,\tilde{Q}_{\zeta}}$
  \begin{equation}
    \sum_{a=1}^n \Ec^{(2)}_a+\tilde{\Ec}^{(2)}_a
    =
    \OO{\varphi^{4Q_{\zeta}}(\Psi[\tilde{\Lambda}])^{2}|m_c|^{-1}}
    .
  \end{equation}
\end{lem}
\begin{proof}
  The main tool in the proof of this lemma is the Abstract decoupling lemma (Theorem~\ref{thm:abs_dec}).
  We consider this lemma in the following setting: let $\Ic=\llbracket 1,nN\rrbracket$, $\Ic_i=\{i(a)\},i\in\llbracket 1,N\rrbracket$, $\QQ_i=\EE_{\yb_{i(a)}}$, and consider the random variables $S_i=(w \LRI{G}{}{a}{ii})^{-1}$.
Then $\LRI{\Zc}{}{a}{i}=\QQ_i S_i$ and it will be enough to show that the conditions \eqref{eq:abs_dec_cond_i}-\eqref{eq:abs_dec_cond_iii} in the hypothesis of the Abstract decoupling lemma hold with $\Xc=\varphi^{4Q_{\zeta}}\Psi[\tilde{\Lambda}]|m_c|^{-1}$ and $\Yc=|m_c|$.

Condition \eqref{eq:abs_dec_cond_ii} can be verified using the uniform subexponential decay condition for the entries of the matrix $X$, and condition \eqref{eq:abs_dec_cond_iii} holds by the assumption of the lemma.
Therefore, we need to show that  decomposition \eqref{eq:abs_dec_cond_i} holds for any subset $\AA\subset\llbracket 1,N\rrbracket$ with $|\AA|\leq p_N$.
It was shown in \cite{BourYauYin} that the existence of such decomposition can be deduced from the following set of estimates holding with $\zeta$-high probability for any $w\in S_{z,\delta,Q_{\zeta}}$:
\begin{EA}{ll}
  (i)&\quad \Lambda\leq \varphi^{-Q_{\zeta}}|w|^{-1/2}
  ,\quad 
  \Psi
  \leq 
  \varphi^{-Q_{\zeta}}|w|^{-1/2}
  ,\\
  (ii)&\quad \max_{i,j}|G_{ij}-m_c \delta_{ij}|
  \leq 
  \varphi^{2Q_{\zeta}}|w|^{-1/2}\left(\frac{|w|^{1/2}}{N\eta}\right)^{1/4}
  ,\\
  (iii)&\quad \mbox{for any }i\neq j\mbox{ and } \UU,\TT\subset \llbracket 1(a),N(a)\rrbracket:\quad |\TT|+|\UU|\leq p_N
  \\
   &\quad |(1-\EE_{\yb_{i(a)}})\yb_{i(a)}^*\Gc^{(i(a)\TT,\emptyset)}\yb_{i(a)}|+|(1-\EE_{\yb_{i(a)}})\yb_{i(a)}^*\Gc^{(i(a)j(a)\TT,\emptyset)}\yb_{j(a)}|\leq \varphi^{2Q_{\zeta}}\Psi
   ,\\
  & \quad |(1-\EE_{y_{i(a)}}) y^{(i(a))}_{i(a)}G^{(i(a),i(a)\UU)}(y^{(i(a))}_{i(a)})^*|+|(1-\EE_{y_{i(a)}}) y^{(i(a))}_{i(a)}G^{(i(a),i(a)j(a)\UU)}(y^{(i(a))}_{j(a)})^*|\leq \varphi^{2Q_{\zeta}}\Psi
  ,\\
  (iv)& \quad \mbox{for any }i\mbox{ and } \TT\subset \llbracket 1(a),N(a)\rrbracket:\quad |\TT|\leq p_N
  \\
  &\quad |\LRI{\Gc}{(i(a)\TT,\emptyset)}{a}{ii}-\frac{1}{w(1+\LRI{m}{(i(a)\TT,\emptyset)}{a+1}{G})}|\leq \varphi^{2Q_{\zeta}}\Psi
  .
\end{EA}
Note that $(i)$ and $(ii)$ follow from Theorem~\ref{thm:continuity_argument}, $(iii)$ follows from Lemma~\ref{lem:lde} and $(iv)$ was proven in Corollary\ref{cor:weak_concentration_entries}.
Therefore, we can repeat the argument used by Bourgade, Yau and Yin in the proof of the strong estimates of the Stieltjes transform of one non-Hermitian matrix to find the decomposition \eqref{eq:abs_dec_cond_i} in our setting.
See \cite[Section 7.2]{BourYauYin} for the detailed proof.
\end{proof}

We can now finish the proof of Theorem~\ref{thm:conc_main}.

Consider first the case when $w$ is far enough from $\lambda_{\pm}$.
From Theorem~\ref{thm:continuity_argument} we know that the initial bound
\begin{equation}
  \Lambda\leq \alpha |m_c| =:\tilde{\Lambda}_0
\end{equation}
holds, therefore we can take 
\begin{equation}
  \Psi[\tilde{\Lambda}_{0}]
  =
  \sqrt{\frac{\Im m_c +|m_c|}{N\eta}}+\frac{1}{N\eta}
  .
\end{equation}
Suppose that $|w-\lambda_{\pm}|\geq M^{-1}$.
Then from propositions \ref{pr:sep} and \ref{pr:large} and lemmas \ref{lem:strong_error_1} and \ref{lem:strong_error_2}
we get that 
\begin{equation}
  \Lambda
  \leq
  \varphi^{4Q_{\zeta}}\left( \frac{\Im m_c + |m_c|}{|m_c|N\eta}+\frac{1}{(N\eta)^2}\right)
  \leq
  \varphi^{4Q_{\zeta}}\frac{1}{N\eta}
\end{equation}

In the case $|w-\lambda_{\pm}|\leq M^{-1}$ the iteration procedure used by Bourgade, Yau and Yin in \cite{BourYauYin} is applicable in our setting.
Note that in this case $|m_c|\sim 1$ and $\Im m_c \sim |w-\lambda_{\pm}|^{1/2}$.
The idea is that if, for example, $ M^{3/2} \varphi^{4Q_{\zeta}}(\Psi[\tilde{\Lambda_0}])^2\leq |w-\lambda_{\pm}|\leq M^{-1}$, then
\begin{equation}
  \Lambda
  \leq
  C \frac{\varphi^{4Q_{\zeta}}(\Im m_c+ \tilde{\Lambda}_0)}{N\eta|w-\lambda_{\pm}|^{1/2}}
  \leq
  C \varphi^{4Q_{\zeta}}\left(\frac{1}{N\eta}+ \frac{\tilde{\Lambda}_0}{N\eta|w-\lambda_{\pm}|^{1/2}}\right)
  .
\end{equation}
But 
\begin{equation}
  \varphi^{4Q_{\zeta}}\frac{\tilde{\Lambda}_0}{N\eta|w-\lambda_{\pm}|^{1/2}}
  \leq
  C\varphi^{4Q_{\zeta}}\frac{\tilde{\Lambda}_0}{N\eta }\frac{\sqrt{N\eta}}{\varphi^{2Q_{\zeta}}\sqrt{\Im m_c +\tilde{\Lambda}_0}}
  \leq
  C\varphi^{2Q_{\zeta}}\sqrt{\frac{\tilde{\Lambda}_0}{N\eta }}
  ,
\end{equation}
and thus
\begin{equation}
  \Lambda
  \leq
  C \left(\varphi^{4Q_{\zeta}}\frac{1}{N\eta}+\varphi^{2Q_{\zeta}}\sqrt{\frac{\tilde{\Lambda}_0}{N\eta }}\right)
  =:\tilde{\Lambda}_1
  \leq
  \tilde{\Lambda}_0
\end{equation}
with probability at least $1-e^{-\varphi^{\zeta}(\log N)^{-2}}$.
Note that $(\log N)^{-1}$ factor in the exponent appears due to \eqref{eq:abs_dec_cond_iii}.
If we repeat the above procedure with the error term $\Psi[\tilde{\Lambda}_0]$, we shall get again a better estimate
\begin{equation}
  \Lambda
  \leq
  C \left(\varphi^{4Q_{\zeta}}\frac{1}{N\eta}+\varphi^{2Q_{\zeta}}\sqrt{\frac{\tilde{\Lambda}_1}{N\eta }}\right)
\end{equation}
that holds with probability $1-e^{-\varphi^{\zeta}(\log N)^{-2}}$.
It was shown in \cite[Section~7.1]{BourYauYin} that if we iterate $K:=\log \log N/\log 2$ times, we shall obtain that
\begin{equation}
  \varphi^{2Q_{\zeta}} \sqrt{\frac{\tilde{\Lambda}_K}{N\eta }}
  \leq
  \frac{\varphi^{4Q_{\zeta}}}{N\eta }
\end{equation}
with probability at least $1-e^{-\varphi^{\zeta/2}}$.
Note that a similar argument applies in the regime when $|w-\lambda_{\pm}|\leq M^{3/2}\varphi^{4Q_{\zeta}}(\Psi[\tilde{\Lambda}_k])$ for some $k\in \{0,1,\ldots,K\}$.
Therefore, we deduce that \eqref{eq:conc_main} holds for all $w\in S_{z,\delta,\tilde{Q}_{\zeta}}$.


\appendix

\section{Proof of Proposition~\ref{pr:Gamma_inv}}

First of all, we can easily verify that if $I-\Gamma_1\Gamma_1^T$ is invertible, then
\begin{equation*}
  \begin{pmatrix}
    I & \Gamma_1
    \\
    \Gamma_1^T & I
  \end{pmatrix}
  ^{-1}
  =
  \begin{pmatrix}
    (I-\Gamma_1 \Gamma_1^{T})^{-1} & -\Gamma_1(I-\Gamma_1^T \Gamma_1)^{-1}
    \\
    -\Gamma_1^T(I-\Gamma_1 \Gamma_1^{T})^{-1} & (I-\Gamma_1^T \Gamma_1)^{-1} 
  \end{pmatrix}
  =
  \begin{pmatrix}
    I & -\Gamma_1
    \\
    -\Gamma_1^T & I
  \end{pmatrix}
  \begin{pmatrix}
    (I-\Gamma_1 \Gamma_1^{T})^{-1} & 0
    \\
    0 & (I-\Gamma_1 \Gamma_1^{T})^{-1} 
  \end{pmatrix}
  ,
\end{equation*}
where in the last equality we used that $\Gamma_1^T \Gamma_1=\Gamma_1 \Gamma_1^T = \mathrm{Circulant}(\gfr_1^2+\gfr_2^2, -\gfr_1\gfr_2,0,\ldots,0,-\gfr_1\gfr_2)$.
Therefore,
\begin{equation}
  \| (I-\Gamma_1 \Gamma_1^T)^{-1}\|_{\infty}
  \leq
  \|\Gamma^{-1}\|_{\infty}
  \leq
  2n \|\Gamma \|_{\infty}\| (I-\Gamma_1 \Gamma_1^T)^{-1}\|_{\infty}
  .
\end{equation}
Since $\gfr_1\sim\gfr_2\sim 1$, we get that all the non-zero entries of the matrix $\Gamma$ are of order $1$.
Thus, we deduce that $\|\Gamma^{-1}\|_{\infty}\sim \|(I-\Gamma_1\Gamma_1^{T})^{-1}\|_{\infty}$.
Note that
$I-\Gamma_1\Gamma_1^{T}=\mathrm{Circulant}(1-\gfr_1^2-\gfr_2^2, \gfr_1\gfr_2,0,\ldots,0,\gfr_1\gfr_2)$.
The following lemma allows us to calculate directly the eigenvalues of $I-\Gamma_1\Gamma_1^T$ and the entries of $(I-\Gamma_1\Gamma_1^T)^{-1}$.
\begin{lem}
  Suppose we have a circulant matrix $C=\mathrm{Circulant}(c_0,c_1,\ldots,c_{n-1})$ and let $l_j(C),1\leq j \leq n$ be eigenvalues of $C$. Then
  \begin{itemize}
  \item[(i)] $    l_j(C)=\sum_{k=0}^{n-1} c_k e^{2\pi\sqrt{-1}jk/n} 
    ,\quad
    j=1,\ldots,n
    ;$
  \item[(ii)](\cite[p. 91]{Aldr}) if $C$ is invertible, then $C^{-1}$ is a circulant matrix with coefficients
    \begin{equation}
      \frac{1}{n}\sum_{k=1}^{n}\frac{1}{l_k(C)}e^{jk\sqrt{-1}2\pi/n}
      ,\quad
      j=1,\ldots,n
    \end{equation}

  \end{itemize}
\end{lem}

Therefore, to study the behaviour of $\|\Gamma^{-1}\|$ it is enough to study the eigenvalues of $I-\Gamma_1 \Gamma_1^{T}$.
From the above formula, for $1\leq j\leq n$
\begin{equation}\label{eq:app_1}
  l_j(I-\Gamma_1 \Gamma_1^T)
  =
  1-\gfr_1^2 -\gfr_2^2 + \gfr_1\gfr_2 (e^{j \sqrt{-1}2\pi/n}+e^{j (n-1)\sqrt{-1}2\pi/n})
  =
  1-\gfr_1^2 -\gfr_2^2 + \gfr_1\gfr_2 2\Re e^{j \sqrt{-1}2\pi/n}
  .
\end{equation}

\paragraph{Case 1.}

We shall show that $l_n(I-\Gamma_1 \Gamma_1^T)$ behaves as $|w-\lambda_{\pm}|^{1/2}$ when $w$ is close to $\lambda_{\pm}$.
Note that $l_n(I-\Gamma_1 \Gamma_1^T)= 1-\gfr_1^2 -\gfr_2^2 +2 \gfr_1\gfr_2 = (1-(\gfr_1-\gfr_2))(1+(\gfr_1-\gfr_2))$.
From the formulas \eqref{eq:prop_mc_lambda_2} we have that if $w=\lambda_{\pm}$, then
\begin{equation}\label{eq:app_2}
  1-(\gfr_1-\gfr_2)
  =
  1-(\frac{1}{\lambda_{\pm}m_c^2(\lambda_{\pm})}-\frac{|z|^2}{\lambda_{\pm}(1+m_c(\lambda_{\pm}))^2})
  =
  0
  .
\end{equation}
Again using \eqref{eq:prop_mc_lambda_2} we can see that if $w$ is in the neighbourhood of $\lambda_{\pm}$, then
\begin{EA}{rl}
  \frac{1}{w m_c^2(w)}
  &=
  \frac{1}{\lambda_{\pm} m_c^2(\lambda_{\pm})}-\frac{2\beta_{\pm} \sqrt{w-\lambda_{\pm}}}{ m_c(\lambda_{\pm})\lambda_{\pm} m_c^2(\lambda_{\pm})} +O(|w-\lambda_{\pm}|)
  \\
  \frac{1}{w (1+m_c(w))^2}
  &=
  \frac{1}{\lambda_{\pm} (1+m_c(\lambda_{\pm}))^2}-\frac{2\beta_{\pm} \sqrt{w-\lambda_{\pm}}}{(1+ m_c(\lambda_{\pm}))\lambda_{\pm} (1+m_c(\lambda_{\pm}))^2} +O(|w-\lambda_{\pm}|)
\end{EA}
where 
\begin{equation}
    \beta_{\pm}
    =
    \sqrt{\frac{8(\pm 1+\afr)^3}{\pm\afr(\pm 3+\afr)^5}}  
\end{equation}
Now it is enough to show that the coefficient near $\sqrt{w-\lambda_{\pm}}$ does not vanish.
Using the exact formula for $m_c(\lambda_{\pm})$ we have that
\begin{equation}\label{eq:app_3}
  \frac{|z|^2}{(1+m_c(\lambda_{\pm}))^3}-\frac{1}{m_c^3(\lambda_{\pm})}
  =
  \frac{(\afr\pm 3)^{3}}{\mp 8 (\afr\pm 1)^{3}}(\mp 8 |z|^2 - (\afr\pm 1)^{3})  
\end{equation}
Therefore we need to show that
\begin{equation}
\sqrt{\frac{(\afr\pm 3)}{8\afr  (\afr\pm 1)^{3}}}|(\afr\pm 1)^{3}\pm 8 |z|^2  |
\end{equation}
is bounded away from zero.
This follows easily from the fact that we consider the case $1+\tau \leq |z|\leq \tau^{-1}$
and the equality
\begin{equation}\label{eq:app_4}
  (\afr \pm 1)^{3}\pm 8|z|^2
  =
  \sqrt{1+8|z|^2}(\sqrt{1+8|z|^2}\pm 1)(\sqrt{1+8|z|^2}\pm 3)
  .
\end{equation}
This concludes the proof of \eqref{eq:Gamma_inv_1}.
\begin{rem}
  If we take $|z|=1+\tau$ then the coefficient near $\sqrt{w-\lambda_+}$ is bounded away from zero, while the coefficient near $\sqrt{w-\lambda_-}$ is of order $\OO{\tau}$ as $\tau\rightarrow 0$.
\end{rem}

\paragraph{Case 2.}

The proof is similar to the Case~1 with $1\leq \afr\leq 3$ and the coefficient near $\sqrt{w-\lambda_+}$
\begin{equation}
\sqrt{\frac{(\afr+ 3)}{8\afr  (\afr+ 1)^{3}}}|(\afr+ 1)^{3}+ 8 |z|^2  |
\geq
C
.
\end{equation}

\paragraph{Case 3.}
Suppose that $|z|\leq 1-\tau$.
Then using \eqref{eq:prop_mc_zero} we get an approximation of $\gfr_1$ and $\gfr_2$ for $|w|\leq \tau$
\begin{EA}{rl}
  \frac{1}{wm_c^2}
  &=
  \frac{1}{-(1-|z|^2)^2}+\frac{\sqrt{-1}(1-2|z|^2)\sqrt{w}}{(1-|z|^2)^4}+O(w)
  ,\\
  \frac{|z|^2}{w(1+m_c)^2}
  &=
  \frac{|z|^2}{-(1-|z|^2)^2}-\frac{\sqrt{-1}|z|^2\sqrt{w}}{(1-|z|^2)^4}+O(w)
\end{EA}

Then for $1\leq j\leq n$

\begin{EA}{rl}
  l_j(I-\Gamma_1\Gamma_1^T)
  =&
      1-\left(\frac{1}{w m_c^2(w)}\right)^2-\left(\frac{|z|^2}{w (1+m_c(w))^2}\right)^2+2\omega_j\frac{|z|^2}{w m_c^2(w)w(1+m_c(w))^2}
    \\
    =&|z|^2\frac{-4+2\omega_j+5|z|^2-4|z|^4+|z|^6}{(1-|z|^2)^4}
    +\sqrt{w}\frac{2\sqrt{-1}}{(1-|z|^2)^6}(1-2|z|^2-|z|^4+2\omega_j|z|^4)
    +O(w)
\end{EA}
where we denoted $\Re e^{j\sqrt{-1}2\pi/n}$ by $\omega_j\in [-1,1]$.
Since $|z|\leq 1-\tau$, we have
\begin{equation}
  \frac{-4+2\omega_j}{\tau^4}
  \leq
  \frac{-4+2\omega_j+5|z|^2-4|z|^4+|z|^6}{(1-|z|^2)^4}
  \leq
  -2+2\omega_j -\tau^2-\tau^3
\end{equation}

If $|z|^2\geq 1/3$, then we can find $\tilde{\tau}>0$ small enough such that for any $1\leq j\leq n$ and $|w|\leq \tilde{\tau}$
\begin{equation}
  |\sqrt{w}\frac{2\sqrt{-1}}{(1-|z|^2)^6}(1-2|z|^2-|z|^4+2\omega_j|z|^4)|
  \leq
  |-2+2\omega_j -\tau^2-\tau^3|\frac{1}{6}
  ,
\end{equation}
so that $|l_j(I-\Gamma_1\Gamma_1^T)|$ is bounded away from zero. 

If $|z|^2\leq 1/3$, then 
\begin{equation}
  1-2|z|^2-|z|^4+2\omega_j|z|^4
  \geq 
  0
  ,
\end{equation}
and this implies that 
\begin{equation}
  \Re \sqrt{w}\frac{2\sqrt{-1}}{(1-|z|^2)^6}(1-2|z|^2-|z|^4+2\omega_j|z|^4)
  \leq
  0
  .
\end{equation}
Thus 
\begin{equation}
  ||z|^2\frac{-4+2\omega_j+5|z|^2-4|z|^4+|z|^6}{(1-|z|^2)^4}+\sqrt{w}\frac{2\sqrt{-1}}{(1-|z|^2)^6}(1-2|z|^2-|z|^4+2\omega_j|z|^4)|
  \sim
  |z|^2 +\sqrt{|w|}
  ,
\end{equation}
which concludes the proof of Case 3.

\paragraph{Case 4.}

Suppose firstly that $|z|\leq \tilde{\varepsilon}$ for some $\tilde{\varepsilon}>0$.
\begin{equation}
      l_j(I-\Gamma_1\Gamma_1^T)
    =
    1-\frac{1}{(wm_c^2)^2}-\frac{|z|^4}{(w(1+m_c)^2)^2}+2\omega_j \frac{|z|^2}{wm_c^2w(1+m_c)^2}
  \end{equation}
Since for $\max\{\lambda_-,0\}+\tilde{\tau}\leq E\leq \lambda_+-\tau$ the imaginary part of $m_c$ is of order $1$, there exists $c>0$ such that 
\begin{equation}
    |1-\frac{1}{(wm_c^2)^2}|
    \geq
    c
\end{equation}
Therefore, if we take $\tilde{\varepsilon}$ so small that
\begin{equation}
    |z|^2|\frac{|z|^2}{(w(1+m_c)^2)^2}-2\omega_j \frac{1}{wm_c^2w(1+m_c)^2}|
    \leq
    c/2
\end{equation}
then we get that $|l_j(I-\Gamma_1\Gamma_1^{T})|>c/2$

Consider now the cases
\begin{itemize}
\item[(i)] $\tilde{\varepsilon}\leq |z|\leq 1-\tau$, $\tilde{\tau}\leq E\leq \lambda_+-\tau$ and $\eta=0$, or
\item[(ii)] $1+\tau\leq |z|\leq \tau^{-1}$, $\lambda_-+\tau\leq E\leq \lambda_+-\tau$ and $\eta=0$ 
\end{itemize}
Suppose that in these cases $|\det(I-\Gamma_1\Gamma_1^T)|\geq c$ for some $c>0$.
Since $|\det(I-\Gamma_1\Gamma_1^T)|$ is a continuous function of $z$ and $w$, using the continuous dependence on $\eta$ we can find $\varepsilon>0$ such that $|\det(I-\Gamma_1\Gamma_1^T)|\geq c/2$ on the sets $\{\tilde{\varepsilon}\leq |z|\leq 1-\tau,\tilde{\tau}\leq E\leq \lambda_+-\tau,0\leq \eta\leq \varepsilon\}$ and $\{1+\tau\leq |z|\leq \tau^{-1},\lambda_-+\tau\leq E\leq \lambda_+-\tau, 0\leq \eta\leq \varepsilon\}$.
Thus, it is enough to show that under the conditions $(i)$ or $(ii)$ the eigenvalues $l_j(I-\Gamma_1\Gamma_1^T),1\leq j\leq n$ do not vanish.

We start with some simplifications. Firstly, from \eqref{eq:mc_equation} we have that
\begin{equation}
  \gfr_2
  =
  \frac{m_c}{1+m_c} \gfr_1 +1
  .
\end{equation}
Thus, we can rewrite the formulas for the eigenvalues of $I-\Gamma_1\Gamma_1^T$ as follows
\begin{EA}{ll}
  l_j(I-\Gamma_1\Gamma_1^T)
  &=
  1-\gfr_1^2 -1-2\frac{m_c}{1+m_c} \gfr_1 -(\frac{m_c}{1+m_c})^2 \gfr_1^2 +2\omega_j \gfr_1(\frac{m_c}{1+m_c} \gfr_1+1)
  \\
  &=
  -\gfr_1(\gfr_1 +2\frac{m_c}{1+m_c}  +(\frac{m_c}{1+m_c})^2 \gfr_1 -2\omega_j (\frac{m_c}{1+m_c} \gfr_1+1))
\end{EA}
Recall that $\gfr_1\sim 1$. 
This means that $j$th eigenvalue of $I-\Gamma_1\Gamma_1^T$ is equal to zero if and only if 
\begin{equation}\label{eq:app_5}
  \gfr_1(1+(\frac{m_c}{1+m_c})^2 -2\omega \frac{m_c}{1+m_c})+2\frac{m_c}{1+m_c}-2\omega 
  =
  0
\end{equation}
for $\omega=\omega_j$.
We are going to show that under conditions $(i)$ or $(ii)$ equation \eqref{eq:app_5} has no solution for $\omega\in[-1,1]$.
Using again the equation \eqref{eq:mc_equation} we have
\begin{EA}{ll}
  \gfr_1(1+(\frac{m_c}{1+m_c})^2& -2\omega \frac{m_c}{1+m_c})+2\frac{m_c}{1+m_c}-2\omega 
  =
  \frac{1}{m_c}\frac{1}{wm_c}+\frac{1}{1+m_c}\frac{1}{w(1+m_c)}-\frac{2\omega}{wm_c(1+m_c)} +2\frac{m_c}{1+m_c}   -2\omega 
  \\
  &=
  \frac{1}{m_c}(-1-m_c+\frac{|z|^2}{w(1+m_c)})+\frac{1}{1+m_c}(1+m_c+\frac{1}{wm_c})\frac{1}{|z|^2}-\frac{2\omega}{wm_c(1+m_c)} +2\frac{m_c}{1+m_c}   -2\omega
  \\
  &=
  \frac{1}{wm_c(1+m_c)}(|z|^2+\frac{1}{|z|^2}-2\omega)-\frac{1}{m_c}-1+\frac{1}{|z|^2} + 2 - \frac{2}{1+m_c}   -2\omega
  \\
  &=
  (\frac{1}{wm_c}-\frac{1}{w(1+m_c)})(|z|^2+\frac{1}{|z|^2}-2\omega)-\frac{1}{m_c}- \frac{2}{1+m_c}+\frac{1}{|z|^2} + 1    -2\omega
\end{EA}
Define
\begin{equation}
  d_1
  :=
  |z|^2+\frac{1}{|z|^2}-2\omega
  ,\quad
  d_2
  :=
  \frac{1}{|z|^2}+1-2\omega
\end{equation}
Then \eqref{eq:app_5} is equivalent to 
\begin{equation}\label{eq:app_6}
  (\frac{1}{wm_c}-\frac{1}{w(1+m_c)})d_1=\frac{1}{m_c}+ \frac{2}{1+m_c}-d_2
\end{equation}
We now take imaginary part of the above equation.
Note, that we consider the case $\eta=0, w=E$.
\begin{equation}
  (\frac{\Im m_c}{E|m_c|^2}-\frac{\Im m_c}{E|1+m_c|^2})d_1=\frac{\Im m_c}{|m_c|^2}+ \frac{2\Im m_c}{|1+m_c|^2}
\end{equation}
If $E\in(\max\{0,\lambda_-\},\lambda_+)$, then $\Im m_c >0$. 
Thus we can divide by $\Im m_c$ and obtain 
\begin{equation}\label{eq:app_7}
    (\frac{1}{E|m_c|^2}-\frac{1}{E|1+m_c|^2})d_1
    =
    \frac{1}{|m_c|^2}+ \frac{2}{|1+m_c|^2}  
\end{equation}
Consider now the real part of \eqref{eq:app_6}
\begin{equation}
   \Re m_c (\frac{1 }{E|m_c|^2}-\frac{1 }{E|1+m_c|^2})d_1-\frac{1}{E|1+m_c|^2}d_1
   =
   \Re m_c(\frac{1 }{|m_c|^2}+ \frac{2}{|1+m_c|^2})+\frac{2}{|1+m_c|^2}-d_2
\end{equation}
Together with \eqref{eq:app_7} we obtain
\begin{equation}\label{eq:app_8}
  -\frac{1}{E|1+m_c|^2}d_1
  =
  \frac{2}{|1+m_c|^2}-d_2
\end{equation}
\eqref{eq:app_7} and \eqref{eq:app_8} give us
\begin{equation}\label{eq:app_9}
  \frac{1}{E|m_c|^2}d_1
  =
  \frac{1}{|m_c|^2}+d_2
\end{equation}
Therefore, we have rewritten equation \eqref{eq:app_5} as a system \eqref{eq:app_8}-\eqref{eq:app_9}.
Together with the real and imaginary parts of \eqref{eq:mc_equation} we obtain the following system of equations
\begin{equation}
  \left\{
    \begin{array}{ll}
      (a)&
      \frac{1}{E|m_c|^2}(d_1-E)
      =
      d_2
      \\
      (b)&
      \frac{1}{E|1+m_c|^2}(d_1+2E)
      =
      d_2
      \\
      (c)&
      2 \Re m_c+1
      =
      \frac{|z|^2}{E|1+m_c|^2}
      \\
      (d)&
      \frac{1}{E|m_c|^2}=1+\frac{|z|^2}{E|1+m_c|^2}
    \end{array}
  \right.
\end{equation}
Note that $d_1\geq 0$ and $E>0$, so that from $(a)$ and $(b)$ we get $d_2>0$ and $d_1-E>0$.
Thus we can rewrite the above system as
\begin{equation}
  \left\{
    \begin{array}{ll}
      (a')&
           E|m_c|^2 
           =
           \frac{d_1-E}{d_2}
      
      \\
      (b')&
           E|1+m_c|^2 
           =
           \frac{d_1+2E}{d_2}
      \\
      (c)&
           2 \Re m_c+1
           =
           \frac{|z|^2}{E|1+m_c|^2}
      \\
      (d)&
           \frac{1}{E|m_c|^2}
           =
           1+\frac{|z|^2}{E|1+m_c|^2}
    \end{array}
  \right.
\end{equation}
If we take the difference of equations $(b')$ and $(a')$ we get $2\Re m_c+1$ on the left-hand side, and applying $(c)$ gives the following equation
\begin{equation}
  (e)\quad
  \frac{|z|^2}{E|1+m_c|^2}
      =
      \frac{3}{d_2}
\end{equation}
Replacing either $(a)$ or $(b)$ by $(e)$ we obtain two systems
\begin{equation}
    \left\{
    \begin{array}{ll}
      (e)&
      \frac{|z|^2}{E|1+m_c|^2}
      =
      \frac{3}{d_2}
      \\
      (a)&
      \frac{1}{E|m_c|^2}
      =
      \frac{d_2}{d_1-E}
      \\
      (c)&
      2 \Re m_c+1
      =
      \frac{|z|^2}{E|1+m_c|^2}
      \\
      (d)&
      \frac{1}{E|m_c|^2}
           =
           1+\frac{|z|^2}{E|1+m_c|^2}
    \end{array}
  \right.
  \mbox{ and }
  \left\{
    \begin{array}{ll}
      (e)&
      \frac{|z|^2}{E|1+m_c|^2}
      =
      \frac{3}{d_2}
      \\
      (b)&
      \frac{1}{E|1+m_c|^2}
      =
      \frac{d_2}{d_1+2E}
      \\
      (c)&
      2 \Re m_c+1
      =
      \frac{|z|^2}{E|1+m_c|^2}
      \\
      (d)&
      \frac{1}{E|m_c|^2}=1+\frac{|z|^2}{E|1+m_c|^2}
    \end{array}
  \right.
\end{equation}
Equations $(e),(a)$ and $(d)$ imply that $d_1$ and $d_2$ satisfy the equation
\begin{equation}
  \frac{d_2}{d_1-E}
      =
      \frac{3}{d_2}+1
      ,
\end{equation}
while $(e)$ and $(b)$ give
\begin{equation}
  \frac{3}{|z|^2 d_2}
      =
      \frac{d_2}{d_1+2E}
      .
\end{equation}
From the definition of $d_1$ and $d_2$ we have that $d_1=d_2+|z|^2-1$.
We end up with a following system
\begin{equation}\label{eq:app_10}
  \left\{
    \begin{array}{l}
      d_2
      =
      \frac{3(1-|z|^2+E)}{2+|z|^2-E}
      \\
      d_2^2-\frac{3}{|z|^2}d_2+\frac{3}{|z|^2}(1-|z|^2-2E)
      =
      0
    \end{array}
    \right.
\end{equation}
We are now going to fix $z$ that satisfies condition $(i)$ or $(ii)$, and we will consider $d_2$ as a function of $E$.
We will show that for any $E$ satisfying condition $(i)$ or $(ii)$ and for any $\omega\in [-1,1]$ this system does not have a solution equal to $|z|^{-2}+1-2w$.

\emph{Case 1: $1+\tau\leq |z|\leq \tau^{-1}$.} 
We fix $z$.
Define following functions
\begin{EA}{l}
  f_1^{\pm}(E):=\frac{1}{2}(\frac{3}{|z|^2}\pm \sqrt{\frac{3}{|z|^2}(\frac{3}{|z|^2}-4(1-|z|^2-2E))})
  ,\\ 
  f_2(E):=-3+\frac{9}{2+|z|^2-E}
  .
\end{EA}
First of all, note that $f_1^{-}$ is decreasing and
\begin{equation}
    f_1^{-}(\frac{1-|z|^2}{2})=0
    ,\quad
    f_2(|z|^2-1)=0
    ,
\end{equation}
and thus the graphs of $f_1^-$ and $f_2$ do not intersect in the right upper quarter-plane.
We now show that at the point of intersection of $f_1^+$ and $f_2$ the value of the functions is strictly larger than $3+|z|^{-2}$.

Let $E_1$ and $E_2$ be the point on $\RR_+$ such that $f_1^+(E_1)=f_2(E_2)=3+|z|^{-2}$.
Then
\begin{equation}
  E_1
  =
  |z|^2-\frac{1}{8|z|^2}(3-\frac{1}{3}),
  \quad
  E_2
  =
  |z|^2+2-\frac{9|z|^2}{6|z|^2+1}
  .
\end{equation}
If $|z|\geq 1+\tau$ then $2-\frac{9|z|^2}{6|z|^2+1}>0$, therefore $E_1<E_2$.
We deduce that at the point of the intersection the value of the functions is strictly larger than $3+|z|^{-2}\geq d_2$.
As a result, we conclude that the system \eqref{eq:app_10} has no solution.

The graph below shows how the functions $f_1^{\pm}$ are situated with respect to the function $f_2$.

\setlength{\unitlength}{1cm}
\begin{figure}
  \centering
  \begin{picture}(10,10)
    \linethickness{0.2mm}
    \put(0,5){\vector(1,0){10}}
    \put(5,0){\vector(0,1){10}}
    \linethickness{0.1mm}
    \put(0,2){\line(1,0){10}}
    \put(8,0){\line(0,1){10}}
    \put(5,9){\line(1,0){3}}
    \put(7.89,5){\line(0,1){4}}
    \put(6.95,5){\line(0,1){4}}
    \put(3,6){\line(1,0){2}}
    \linethickness{0.3mm}
    \qbezier(8.2,0)(8.6,1.4)(10,1.8)
    \qbezier(0,2.1)(8,2)(7.9,10)
    \qbezier(3,6)(3,8)(10,10)
    \qbezier(3,6)(3,4)(10,2)
    \put(8.1,4.7){$2+|z|^2$}
    \put(6.6,5.1){$E_1$}
    \put(7.6,5.1){$E_2$}
    \put(3.8,9){$3+\frac{1}{|z|^2}$}
    \put(4,3){$f_2$}
    \put(2.5,6){$f_1^{\pm}$}
    \put(9.7,4.7){$E$}
    \put(5,6){$\frac{3}{2|z|^2}$}
    \put(5,2){$-3$}
  \end{picture}

  \caption{Position of the graphs of the functions $f_1^{\pm}$ and $f_2$ with respect to each other for $|z|>1$.}
\end{figure}
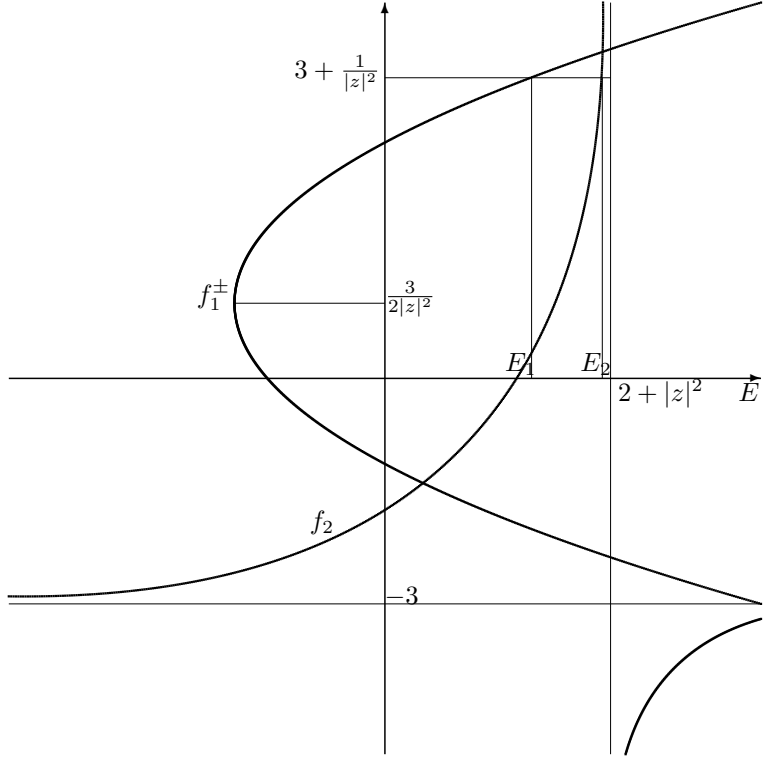
  
\emph{Case 2: $\tilde{\tau}\leq |z|\leq 1- \tau$.} 
We fix $z$ and we define $f_1^{\pm}(E)$ and $f_2(E)$ as above.
As in the Case~1, we start with the zeroes of the functions $f_1^-$ and $f_2$
  \begin{equation}
    f_1^{-}(\frac{1-|z|^2}{2})=0
    ,\quad
    f_2(|z|^2-1)=0
    .
\end{equation}
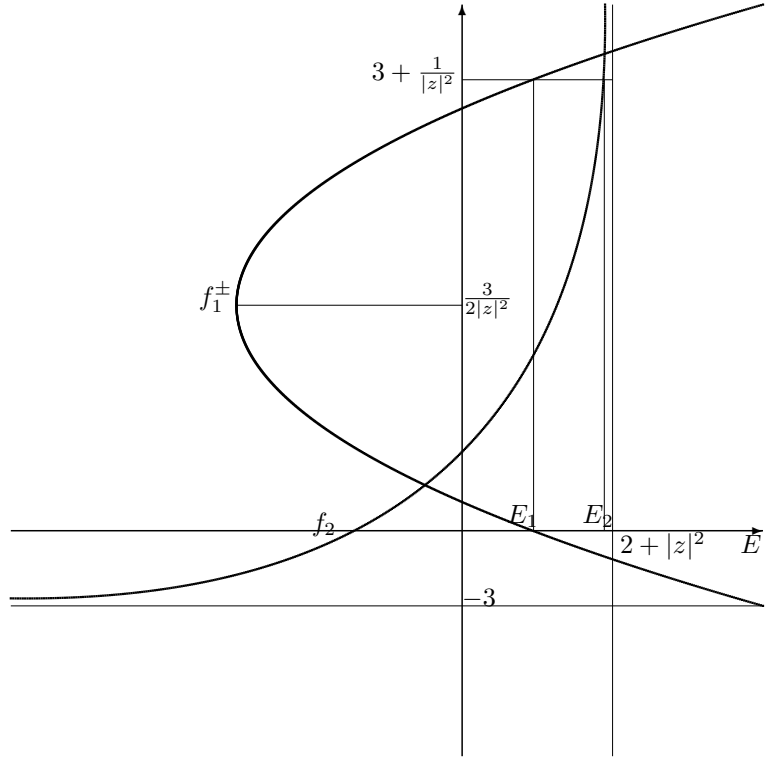
\begin{figure}
  \centering
  \begin{picture}(10,10)
    \linethickness{0.2mm}
    \put(0,3){\vector(1,0){10}}
    \put(6,0){\vector(0,1){10}}
    \linethickness{0.1mm}
    \put(0,2){\line(1,0){10}}
    \put(8,0){\line(0,1){10}}
    \put(6,9){\line(1,0){2}}
    \put(7.89,3){\line(0,1){6}}
    \put(6.95,3){\line(0,1){6}}
    \put(3,6){\line(1,0){3}}
    \linethickness{0.3mm}
    \qbezier(0,2.1)(8,2)(7.9,10)
    \qbezier(3,6)(3,8)(10,10)
    \qbezier(3,6)(3,4)(10,2)
    \put(8.1,2.7){$2+|z|^2$}
    \put(6.6,3.1){$E_1$}
    \put(7.6,3.1){$E_2$}
    \put(4.8,9){$3+\frac{1}{|z|^2}$}
    \put(4,3){$f_2$}
    \put(2.5,6){$f_1^{\pm}$}
    \put(9.7,2.7){$E$}
    \put(6,6){$\frac{3}{2|z|^2}$}
    \put(6,2){$-3$}
  \end{picture}

  \caption{Position of the graphs of the functions $f_1^{\pm}$ and $f_2$ with respect to each other for $|z|<1$.}
\end{figure}
The graphs of these functions intersect in the upper half-plane, but $f_1^-(0)<f_2(0)$, thus we deduce that the coordinate of the intersection is in $\RR_-$. We now need to show that at the point of the intersection of the functions $f_1^+$ and $f_2$ the value of these function is strictly larger than $3+|z|^{-2}$.

Again, define the level point $E_1$ and $E_2$ such that $f_1^+(E_1)=f_2(E_2)=3+|z|^{-2}$, and note that $E_1<E_2$.
This completes the proof in the Case~2.
The graph below shows the functions $f_1^{\pm}$ and $f_2$.

\paragraph{Acknowledgements.}

I would like to thank my PhD advisors Mireille Capitaine and Michel Ledoux for introducing the problem to me, fruitful discussions and reading the manuscript.

\bibliographystyle{plain}
\bibliography{bib}

\begin{thebibliography}{10}

\bibitem{AjanErdoKrug1}
Oskari Ajanki, Laszlo Erdos, and Torben Kr\"{u}ger.
\newblock {Quadratic vector equations on complex upper half-plane}, 2015.
\newblock arXiv:1506.05095.

\bibitem{AjanErdoKrug2}
Oskari Ajanki, Laszlo Erdos, and Torben Kr\"{u}ger.
\newblock {Universality for general Wigner-type matrices}, 2015.
\newblock arXiv:1506.05098.

\bibitem{AkemBurd}
Gernot Akemann and Zdzislaw Burda.
\newblock Universal microscopic correlation functions for products of
  independent {G}inibre matrices.
\newblock {\em J. Phys. A}, 45(46):465201, 18, 2012.

\bibitem{Aldr}
R.~Aldrovandi.
\newblock {\em Special matrices of mathematical physics}.
\newblock World Scientific Publishing Co., Inc., River Edge, NJ, 2001.
\newblock Stochastic, circulant and Bell matrices.

\bibitem{BaiSilv}
Zhidong {Bai} and Jack~W. {Silverstein}.
\newblock {\em Spectral analysis of large dimensional random matrices}.
\newblock Springer Series in Statistics. Springer, New York, second edition,
  2010.

\bibitem{BordGuio}
Charles {Bordenave} and Alice {Guionnet}.
\newblock Localization and delocalization of eigenvectors for heavy-tailed
  random matrices.
\newblock {\em Probability Theory and Related Fields}, 157(3-4):885--953, 2013.

\bibitem{BourYauYin}
Paul {Bourgade}, Horng-Tzer {Yau}, and Jun {Yin}.
\newblock {Local circular law for random matrices}.
\newblock {\em Probab. Theory Related Fields}, 159(3-4):545--595, 2014.

\bibitem{BourYauYin2}
Paul Bourgade, Horng-Tzer Yau, and Jun Yin.
\newblock The local circular law {II}: the edge case.
\newblock {\em Probab. Theory Related Fields}, 159(3-4):619--660, 2014.

\bibitem{BurdJaniWacl}
Z.~Burda, R.~A. Janik, and B.~Waclaw.
\newblock Spectrum of the product of independent random {G}aussian matrices.
\newblock {\em Phys. Rev. E (3)}, 81(4):041132, 12, 2010.

\bibitem{Gini}
Jean Ginibre.
\newblock Statistical ensembles of complex, quaternion, and real matrices.
\newblock {\em J. Mathematical Phys.}, 6:440--449, 1965.

\bibitem{GotzTikh}
Friedrich {G{\"o}tze} and Alexander {Tikhomirov}.
\newblock The circular law for random matrices.
\newblock {\em Ann. Probab.}, 38(4):1444--1491, 2010.

\bibitem{HornJohn}
Roger~A. {Horn} and {Johnson}~Charles R.
\newblock {\em Matrix analysis}.
\newblock Cambridge University Press, Cambridge, second edition, 2013.

\bibitem{Mcdi}
Colin McDiarmid.
\newblock On the method of bounded differences.
\newblock In {\em Surveys in Combinatorics, 1989 ({N}orwich, 1989)}, number 141
  in London Math. Soc. Lecture Note Ser., pages 148--188. Cambridge University
  Press, 1989.

\bibitem{NemiNHPO}
Yuriy Nemish.
\newblock No outliers in the spectrum of the products of independent
  non-hermitian random matrices with independent entries.
\newblock arXiv:1412.2410, 2014.

\bibitem{OrouSosh}
Sean {O'Rourke} and Alexander {Soshnikov}.
\newblock Products of independent non-hermitian random matrices.
\newblock {\em Electron. J. Probab.}, 16:no. 81, 2219--2245, 2011.

\bibitem{PastShch}
Leonid Pastur and Mariya Shcherbina.
\newblock {\em Eigenvalue distribution of large random matrices}, volume 171 of
  {\em Mathematical Surveys and Monographs}.
\newblock American Mathematical Society, Providence, RI, 2011.

\bibitem{PillYin}
Natesh~S. {Pillai} and Jun {Yin}.
\newblock Universality of covariance matrices.
\newblock {\em Ann. Appl. Probab.}, 24(3):935--1001, 2014.

\bibitem{SaffToti}
Edward~B. Saff and Vilmos Totik.
\newblock {\em Logarithmic potentials with external fields}, volume 316 of {\em
  Grundlehren der Mathematischen Wissenschaften [Fundamental Principles of
  Mathematical Sciences]}.
\newblock Springer-Verlag, Berlin, 1997.
\newblock Appendix B by Thomas Bloom.

\bibitem{TaoVuKris}
Terence Tao and Van Vu.
\newblock Random matrices: universality of {ESD}s and the circular law.
\newblock {\em Ann. Probab.}, 38(5):2023--2065, 2010.
\newblock With an appendix by Manjunath Krishnapur.

\bibitem{TaoVuNHLoc}
Terence Tao and Van Vu.
\newblock Random matrices: {U}niversality of local spectral statistics of
  non-{H}ermitian matrices.
\newblock {\em Ann. Probab.}, 43(2):782--874, 2015.

\bibitem{Yin}
Jun Yin.
\newblock The local circular law {III}: general case.
\newblock {\em Probab. Theory Related Fields}, 160(3-4):679--732, 2014.

\end{thebibliography}
\noindent
\\
\\
\textsc{Yuriy Nemish,}
\\
\textsc{\small{Institut de Mathématiques de Toulouse,
UMR 5219 du CNRS\\
Université de Toulouse, F-31062, Toulouse, France}}
\\
\textit{E-mail address:}
\texttt{yuriy.nemish@math.univ-toulouse.fr}

\end{document}